\documentclass[11pt, reqno]{amsart}
\usepackage{amssymb, amsfonts, amsmath, amsthm, tikz, tikz-cd}
\usepackage{amsrefs}
\usetikzlibrary{decorations.pathmorphing}
\usetikzlibrary{trees}
\usepackage{faktor, graphicx, xcolor}
\usepackage{float}
\newtheorem{theorem}{Theorem}[section]
\newtheorem*{theorem*}{Theorem}
\newtheorem{lemma}[theorem]{Lemma}

\newtheorem{proposition}[theorem]{Proposition}
\newtheorem{proposition-and-definition}[theorem]{Proposition and Definition}
\newtheorem{corollary}[theorem]{Corollary}

\theoremstyle{definition}
\newtheorem{definition}[theorem]{Definition}
\newtheorem{notation}[theorem]{Notation}
\newtheorem{remark}[theorem]{Remark}
\newtheorem{notation-and-remark}[theorem]{Notation and Remark}
\newtheorem{notation-and-definition}[theorem]{Notation and Definition}
\newtheorem{definition-and-remark}[theorem]{Definition and Remark}
\newtheorem{remark-and-notation}[theorem]{Remark and Notation}
\newtheorem{remark-and-definition}[theorem]{Remark and Definition}
\newtheorem{example}[theorem]{Example}

\newtheorem{ad-hoc-item}[theorem]{  }

\numberwithin{equation}{section}

\newcommand{\cA}{ {\mathcal A} }

\newcommand{\cL}{ {\mathcal L} }

\newcommand{\cP}{ {\mathcal P} }

\newcommand{\cS}{ {\mathcal S} }

\newcommand{\cU}{ {\mathcal U} }  
\newcommand{\cW}{\mathcal{W}}

\newcommand{\piR}{\pi_{\text{rain}}}

\newcommand{\bC}{ {\mathbb{C}} }

\newcommand{\bN}{ {\mathbb{N}} }

\newcommand{\eval}{ \mathrm{eval} }

\newcommand{\ee}{ \varepsilon }

\newcommand{\ui}{ \underline{i} }
\newcommand{\uj}{ \underline{j} }
\newcommand{\uk}{ \underline{k} }
\newcommand{\uee}{\underline{\ee}}

\newcommand{\neutral}{\cW_{0}(d,n)}
\usepackage[colorlinks=true]{hyperref}

\title{A Central Limit Theorem in the framework of the Thompson Group $F$}
\author{Arundhathi Krishnan}
\address{A.~Krishnan, Department of Mathematics and Computer Studies, Mary Immaculate College, St. Patrick’s Campus, Thurles, Ireland}
\email{Arundhathi.Krishnan@mic.ul.ie}

\begin{document}

\begin{abstract}
We discuss a central limit theorem in the framework of the group algebra of the Thompson group $F$. We consider the sequence of self-adjoint elements given by $a_n=\frac{g_n+g_n^{*}}{\sqrt{2}}$ in the noncommutative probability space $(\bC(F),\varphi)$, where the expectation functional $\varphi$ is the trace associated to the left regular representation of $F$, and the $g_n$-s are the generators of $F$ in its standard infinite presentation. We show that the limit law of the sequence $s_n = \frac{a_0+\cdots+a_{n-1}}{\sqrt{n}}$ is the standard normal distribution.
\end{abstract}
\maketitle
\section{Introduction and Preliminaries}\label{section:intro}

It is well-known that simplified versions of certain central limit theorems --  for instance, the classical and free versions -- can be proved algebraically. Let $(\cA,\varphi)$ be a $*$-probability space, $(a_n)$ be a sequence of self-adjoint elements in $\cA$, and $s_n= \frac{1}{\sqrt{n}}(a_0+\cdots +a_{n-1})$ for each $n \in \bN$. The moment of order $d$ of the element $s_n$ is given by $\varphi(s_n^d)$.
We are often interested in the existence of the limit of the sequence $(\varphi(s_n^d))$, and its value if it exists, for each $d \in \bN$:
\begin{equation}\label{equation:limitmoment}
\lim_{n \to \infty} \varphi(s_n^d)= \lim_{n \to \infty} \frac{1}{n^{d/2}}
\sum_{ \ui : [d] \to \{ 0, \ldots , n-1 \} }   
\ \varphi( 
a_{\ui (1)} \cdots a_{\ui (d)}).
\end{equation}

Usually, the sum on the right hand side of \eqref{equation:limitmoment} is rewritten using some property of the sequence $(a_n)$ so that the number of terms being summed no longer depend on the value of $n$. For example, in the above mentioned cases of the classical and free central limit theorems respectively, the sum is taken over all partitions, and non-crossing pair partitions respectively, of $\{1,\ldots, d\}$ (see \cite[Theorem 3]{Sp90} and \cite[Lecture 8]{NS06}). Henceforth, we will denote the set $\{1,\ldots,d\}$ by $[d]$ for $d \in \bN$.

Algebraic central limit theorems have been studied in various contexts (see \cite{Sk22} for a nice overview). In the setting of the infinite symmetric group $S_{\infty}$, Biane showed that the law of a normalized sequence of random variables coming from the star generators converges to the law of a semi-circular system (see \cite[Theorem 1]{Bi95}). Further, K\"ostler and Nica, and Campbell, K\"ostler and Nica showed that the limit distribution of a sequence coming from the star generators is connected to the average empirical eigenvalue distribution of a random GUE matrix (See \cite[Theorem 1.1]{KN21} and \cite[Theorem 2.9]{CKN22}). 

We are interested in a limit theorem in which the sequence $(a_n)$ comes from the generators of the Thompson group $F$. The standard infinite presentation of $F$ is as follows:
\[
F= \langle g_0, g_1, \ldots \mid g_n g_k = g_k g_{n+1}, \ 0\leq k < n <\infty  \rangle .
\]

The Thompson group $F$ is one of three groups $F, V$ and $T$ introduced by Richard Thompson in 1965. It can be described as a certain subgroup of the group of piece-wise linear homeomorphisms on the unit interval. Many of its unusual properties have been studied since then (see \cite{CFP96} and \cite{CF11}), in particular, due to the still open question of its non-amenability. Several aspects of its connections to subfactor theory and noncommutative stochastic processes have been studied, for instance in \cites{BJ19a, Br20, AJ21, KK22}.

Consider the noncommutative probability space $(\bC(F), \varphi)$ where $\bC(F)$ is the group $*$-algebra of $F$ and the expectation functional $\varphi$ is given by the trace associated to the left regular representation of $F$. That is, with $e$ representing the identity element of $F$,
\[\varphi(x) = \begin{cases}
1, & x=e\\
0, & x \neq e.
\end{cases}
\]

Each element $g$ of $F$ can be viewed as an element of the group algebra $\bC(F)$; it is also written as $g$, and is unitary, with $g^* = g^{-1}$.
Let $(a_n)$ be the sequence of self-adjoint elements of $\bC(F)$ defined by 
\[a_n := \frac{g_n +g_n^*}{\sqrt{2}}, \ n \in \bN_0. \] 

Note that $\varphi(a_n)=0$ and $\varphi (a_n^2) = 1$ for each $n \in \bN_0$, so that $(a_n)$ is a sequence of identically distributed, self-adjoint random variables which are centered and have variance $1$ in the probability space $(\bC(F),\varphi)$.
Before stating our main theorem, we remind the reader of the definition of a non-commutative probability space, and the notion of convergence in distribution of a sequence of random variables. 

\begin{definition}[Non-commutative probability space] 
    A non-commutative probability space $(\cA,\varphi)$ consists of a unital $*$-algebra over $\bC$ and a unital positive linear functional $\varphi$ on $\cA$. The elements $a \in \cA$ are called non-commutative random variables in $(\cA,\varphi)$.
\end{definition}

\begin{definition}[Convergence in distribution]
Let $(\cA_n,\varphi_n) \ (n \in \bN)$ and $(\cA,\varphi)$ be non-commutative probability spaces and consider random variables $a_n \in \cA_n$ for each $n \in \bN$, and $a \in \cA$. We say that $(a_n)$ converges in distribution to $a$ as $n \to \infty$, and denote this by
\[
a_n \stackrel{\text{distr}}{\longrightarrow} a,
\]
if we have
\[
\lim_{n \to \infty} \varphi_n(a_n^d) = \varphi(a^d), \ \forall d \in \bN.
\]
(See Definition 8.1 and Remarks 8.2 in \cite{NS06} for some background on this definition and how it relates to weak convergence.)
\end{definition}

We now state our main result -- a central limit theorem in the framework of the Thompson group $F$.

\begin{theorem} [CLT for the sequence $a_n$]\label{theorem:main}
Let  $(a_n)$ be the sequence of self-adjoint random variables in $(\bC(F), \varphi)$ given by
\[
a_n = \frac{g_n+g_n^*}{\sqrt{2}}, \ n \in \bN_0
\]
and
\[
s_n := \frac{1}{\sqrt{n}} (a_0+ \cdots +a_{n-1}), \ n \in \bN.    
\] 

Then we have
\[\lim_{n \to \infty} \varphi (s_n^d) = \begin{cases}
(d-1)!! & \ \text{for } d \text{ even,} \\
0 & d \ \text{for } d \text{ odd.}
\end{cases}\]

That is,
\[
s_n \stackrel{\text{distr}}{\longrightarrow} x,
\]
where $x$ is a normally distributed random variable of variance $1$.
\end{theorem}

The normal distribution is determined by its moments.
As the random variables $s_n$ have moments of all orders, 
we can invoke Theorem 
30.2 of \cite{Bil95} to conclude that
the probability measures $\mu_n$ corresponding to the the random variables $s_n$ converge weakly to $\mu= N(0,1)$, where $N(0,1)$ is the centered normal distribution of variance $1$.

\subsection*{Methodology used to compute moments}
Given a positive integer $d$, the
moment of order $d$ of $s_n$ can be expressed in terms of the expectation functional $\varphi$ evaluated on products of the generators of $F$ and their inverses as follows:
\begin{equation}\label{equation:moment}
\hspace{1cm}
\varphi (s_n^d) = \frac{1}{(2n)^{d/2}}
\sum_{ \substack{\ui : [d] \to \{ 0, \ldots , n-1 \}, \\
                  \uee : [d] \to \{ -1,1 \} }  } 
\ \varphi \Bigl( 
g_{\ui (1)}^{\uee (1)} \cdots g_{\ui (d)}^{\uee (d)} \Bigr).
\end{equation}

We will refer to tuples $(\ui,\uee)$ with $\ui: [d] \to \bN_0$ and $\uee: [d] \to \{-1,1\}$ as \emph{words} of length $d$. For $w=(\ui,\uee)$, we write $\eval_F(w)$ to mean the element $g_{\ui(1)}^{\uee(1)}\cdots g_{\ui(d)}^{\uee(d)}\in F$. We call a word $w$ \emph{neutral} if $\eval_F(w)=e$.
It is clear from the expansion of the moments of $s_n$ in \eqref{equation:moment} and the definition of the trace $\varphi$, that for each $d \in \bN$, evaluating $\lim_{n \to \infty} \varphi(s_n^d)$ reduces to counting the number of neutral words $(\ui,\uee)$ of length $d$ with $\ui$ taking values in $\{0,\ldots, n-1\}$. 

 Our theorem does not fall into the category arising from the Speicher -- von Waldenfels general algebraic central limit theorem, as the sequence $(a_n)$ is not exchangeable, and does not satisfy the singleton vanishing property (see \cite{BS94} and \cite{SW94}). The sequence $(a_n)$ is also easily seen not to be spreadable. However, we will still use the combinatorics of pair partitions to compute our moments as there turns out to be a natural way to associate a pair partition to each neutral word.

It is evident from the length-preserving relations of $F$ that no word $(\ui,\uee)$ of odd length can satisfy $\eval_F((\ui,\uee))=e$. Our approach will be to show that for every even integer $d$,  each neutral word of length $d$ corresponds to a unique pair partition of $[d]$. In order to show this, we will use the language of so-called abstract reduction systems to show that any neutral word of length $d$ in $F$ has a unique normal form, also of length $d$, of the following type:
\[(g_{\uj(1)},\ldots,g_{\uj(\frac{d}{2})},g_{\uj(\frac{d}{2})}^{-1},\ldots, g_{\uj(1)}^{-1}) \ \text{ with }\  \uj(1) +1\geq \uj(2), \ldots, \uj(\frac{d}{2}-1)+1 \geq \uj(\frac{d}{2}).\]

The pattern of this normal form allows us to naturally pair up generators of $F$ and their inverses. Tracing back the steps that transform a word to its unique normal form allows us to assign a permutation $\tau$ in $\cS_d$ and thereby, a pair partition $\pi$ of $[d]$, to the word $(\ui,\uee)$.
We will then arrive at the following expression for the large $n$ limit of the moment of order $d$ of $s_n$:

\begin{equation}\label{equation:pairpart}
\lim_{n \to \infty} \varphi(s_n^d) =  \sum_{ \pi \in \cP_2(d)} \lim_{n\to \infty}  \frac{1}{(2n)^{d/2}}\sum_{\substack{\tau \in \cS_d, \\ \tau(\pi) =\piR}} |\cW_0(d,n,\tau)|. 
\end{equation}

Here, $\piR$ is the so-called rainbow pair-partition, $\tau(\pi) = \piR$ means that the pair partition $\pi$ is transformed to the rainbow pair-partition $\piR$ via the permutation $\tau$, and $|\cW_0(d,n,\tau)|$ denotes the number of neutral words $(\ui,\uee)$ of length $d$ with letters coming from $\{g_0,\ldots, g_{n-1}\}$, and whose assigned permutation is $\tau$. We will show that for every permutation $\tau$, this number is sandwiched between polynomials in $n$, of degree $\frac{d}{2}$, and with leading coefficient $\frac{1}{(\frac{d}{2})!}$. This will allow us to show that each pair partition's contribution is $1$ to the outer sum in \ref{equation:pairpart}. Indeed, Theorem \ref{theorem:main} will then follow, as the number of pair partitions of the set $[d]$ is $0$ for odd $d$, and $(d-1)!!$ for even $d$.

\subsection*{Organisation of the paper}
We are left to outline the contents of this paper. Including the introduction, which forms Section \ref{section:intro}, the paper consists of five sections. In Section \ref{section:normal}, we introduce the set of words in the Thompson group $F$, and an abstract reduction system which will provide the framework for our counting results. In particular, the reduction system gives us a normal form for each neutral word as described above. The language of abstract reduction systems is immensely helping in showing that the normal form obtained is unique.

In Section \ref{section:algorithm}, we recap the basic definitions of pair partitions, and discuss how to assign a pair partition to each neutral word. In other words, we assign a ``bin'' for each neutral-word in a natural way. As an intermediate step, we utilize a permutation $\tau$ in $\cS_d$, where $\cS_d$ is the set of permutations of the set $[d]$. 

In Section \ref{section:count}, we approximate the size of each ``bin''. That is, for each pair partition $\pi$ of $[d]$, we find upper and lower bounds for the number of neutral words formed with letters from the set $\{g_0,\ldots, g_{n-1}, g_0^{-1}, \ldots, g_{n-1}^{-1}\}$ which are assigned to $\pi$ under the binning procedure described in Section \ref{section:algorithm}. 
Finally, in Section \ref{section:main}, we prove our main theorem -- a  central limit theorem for a sequence coming from the generators of the Thompson group $F$.



\section{A normal form for words in \texorpdfstring{$F$}{F}}\label{section:normal}
In this section, we briefly describe $\cW(d)$, the set of words in $F$ of length $d$, and then devise an abstract reduction system on $\cW(d)$ to obtain a particular normal form for each word.

\subsection{The Thompson group \texorpdfstring{$F$}{F}}\label{subsection:Thompson}
In order to arrive at a central limit theorem, we remind that we work with the infinite presentation of $F$:
\[
F = \langle g_0, g_1, g_2, \ldots \mid g_lg_k = g_k g_{l+1}, \ 0\leq k<l<\infty\rangle.
\]

\begin{definition}[Words in $F$] \label{definition:words}
For $d\in \bN$, a word of length $d$ in $F$ is a tuple $(g_{\ui(1)}^{\uee(1)},\ldots, g_{\ui(d)}^{\uee(d)})$ with $\ui: [d] \to \bN_0$ and $\uee: [d] \to \{-1, 1\}$.
\end{definition}

    \begin{itemize}
        \item[-] For brevity, we write $(\ui,\uee)$ to mean $(g_{\ui(1)}^{\uee(1)},\ldots, g_{\ui(d)}^{\uee(d)})$.
        \item[-] We will denote the set of words of length $d$ by $\cW(d)$.
        \item[-]  For a word $w= (\ui,\uee) \in \cW(d)$, we will write $\eval_F(w)$ to mean the element 
        \[g_{\ui(1)}^{\uee(1)}\cdots g_{\ui(d)}^{\uee(d)} \in F.\] 
        \item[-] If $\eval_F(w)=e$, then we call $w$ a \emph{neutral} word.
        \item[-]  A word of length $1$ simply evaluates to a generator or the inverse of a generator, and is called a \emph{letter}.
        \item[-]  Two words $w_1 = (g_{\ui(1)}^{\uee(1)},\ldots, g_{\ui(d_1)}^{\uee(d_1)}) $ and $w_2 = (g_{\ui'(1)}^{\uee'(1)},\ldots, g_{\ui'(d_2)}^{\uee'(d_2)})$ can be concatenated to give
$w = w_1 w_2  := (g_{\ui(1)}^{\uee(1)},\ldots, g_{\ui(d_1)}^{\uee(d_1)}, g_{\ui'(1)}^{\uee'(1)},\ldots, g_{\ui'(d_2)}^{\uee'(d_2)})$.
\item[-]We will call $w$ a \emph{sub-word} of $w'$ if there exist words $w_1$ and $w_2$ such that $w'= w_1 w w_2$
    \end{itemize}

\subsection{Abstract reduction systems and Newman's lemma}
We establish some terminology as used in the theory of abstract reduction systems. We use notations and definitions from \cite[Chapter 2]{BN99}. 

\begin{notation-and-definition} \label{notationdefinition:ARS}
An \emph{abstract reduction system} is a pair $(A, \rightarrow)$, where $A$ is a non-empty set and $\rightarrow$ is a subset of $A \times A$ known as a binary relation, or as a \emph{reduction} in the parlance of theoretical computer science and logic. It is common to write $x \rightarrow y$ rather than $(x,y) \in \rightarrow$, denoting that $x$ is \emph{rewritten} as or \emph{reduced} to $y$.
\begin{enumerate}
    \item The symbol $\stackrel{i}{\rightarrow}$ is used to denote the $i$-fold composition of $\rightarrow$ for $i \in \bN$.
    \item The reflexive transitive closure of $\rightarrow$ is denoted by $\stackrel{*}{\rightarrow}$. We recall here that the reflexive transitive closure of the relation $\rightarrow$ is defined as the smallest relation $\stackrel{*}{\rightarrow}$ containing $\rightarrow$, which is reflexive ($x \stackrel{*}{\rightarrow} x$ for all $x \in A$), and transitive ($x \stackrel{*}{\rightarrow}y, y \stackrel{*}{\rightarrow} z \implies x \stackrel{*}{\rightarrow}z$ for $x,y,z \in A$).
    \item An element $x\in A$ is called \emph{reducible} if there exists some $y \in A$ such that $x \rightarrow y$; otherwise it is called irreducible or said to be in \emph{normal form}. The element $y$ is a normal form of $x$ if $x \stackrel{*}{\rightarrow} y$ and $y$ is in normal form.
    \item Two elements $x, y \in A$ are said to be \emph{joinable} if there exists some $z\in A$ with $x \stackrel{*}{\rightarrow} z$ and $y \stackrel{*}{\rightarrow} z$. This property is denoted by $x \downarrow y$.
\item A reduction $\rightarrow$ is said to be \emph{confluent} if for all $w, x, y\in A$ with $w \stackrel{*}{\rightarrow} x$ and $w \stackrel{*}{\rightarrow} y$, we have $x \downarrow y$. It is said to be \emph{locally confluent} if for all $w, x, y\in A$ with $w \rightarrow x$ and $w \rightarrow y$, we have $x \downarrow y$.
\item A reduction is said to be \emph{terminating} if there is no infinite chain $x_0 \rightarrow x_1 \rightarrow x_2 \rightarrow \cdots$.
\item A reduction is said to be \emph{normalizing} if every element has a normal form.
\item A reduction $\rightarrow$ is called finitely branching if every element has finitely many direct successors.
\end{enumerate}
 \end{notation-and-definition}

 Confluence and local confluence can be visualized pictorially as diamonds as seen in Figure \ref{figure:conflocconf}. Here, solid lines denote universal quantifiers and dashed arrows denote existential quantifiers.

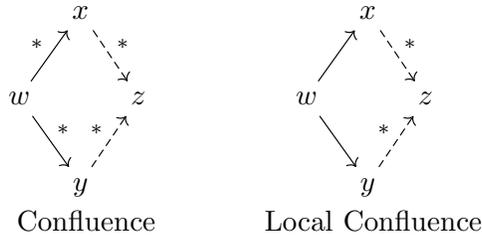
\begin{figure}[H]
    \centering
   \[\begin{array}{cccc}
 \begin{tikzcd}[column sep=tiny]
& x \ar[dr, "*", dashrightarrow] 
& \\
 w \ar[ur, "*", rightarrow] \ar[dr,"*", rightarrow]
    &
      & z
&  \\
& y  \ar[ur, "*", dashrightarrow]
& 
\end{tikzcd}   
&&&
 \begin{tikzcd}[column sep=tiny]
& x \ar[dr, "*", dashrightarrow] 
& \\
 w \ar[ur, rightarrow] \ar[dr,rightarrow]
    &
      & z
&  \\
& y  \ar[ur, "*", dashrightarrow]
&
\end{tikzcd} \\
\text{ Confluence } & & & \text{ Local Confluence } \\
\end{array}
\]
\caption{Visual Representation of Confluence and Local Confluence} \label{figure:conflocconf}
\end{figure}

\begin{remark} \label{remark:normalizing}
It is clear from Notation and Definition \ref{notationdefinition:ARS} that if a reduction is terminating, then every element has a normal form, and hence the reduction is normalizing.    
\end{remark}

We will also state, for completeness, some standard results in the theory of abstract reduction systems.
The following lemma describes when a normalizing reduction results in a \emph{unique} normal form for every element.

\begin{lemma}(\cite[Lemma 2.1.8]{BN99}) \normalfont \label{lemma:uniquenormalform}
If a reduction is normalizing and confluent, then every element has a \emph{unique} normal form.
\end{lemma}

The following characterization is sometimes useful to show that a reduction $\rightarrow$ is terminating.

\begin{lemma}(\cite[Lemma 2.3.3]{BN99}) 
\normalfont \label{lemma:terminating}
A finitely branching reduction terminates if and only if there is a monotone embedding into $(\bN_0,>)$.    
\end{lemma}

The following lemma is due to Newman and  is called the ``diamond'' lemma. See Figure \ref{figure:conflocconf} for an illustrative explanation of the name.

\begin{lemma}(\cite{Ne42}, \cite[Lemma 2.7.2]{BN99}) \normalfont 
\label{lemma:newman}
A terminating reduction is confluent if and only if it is locally confluent.
\end{lemma}

Newman's lemma gives that a terminating reduction which is \emph{locally confluent} in fact has the stronger property of being \emph{confluent}. Altogether, Remark \ref{remark:normalizing}, Lemmas \ref{lemma:uniquenormalform} and \ref{lemma:newman} show that local confluence guarantees the existence of a \emph{unique} normal form for every element in a terminating abstract reduction system.

Two reductions can can be composed to give a new reduction.

\begin{definition} \label{definition:compositionARS}
Let $R \subseteq A \times B$ and $S \subseteq B \times C$ be two reductions. Their composition is defined by
 \[
 S \circ R : = \{(x,z) \mid \exists y \in B \text{ with } (x,y) \in R, (y,z) \in S\}. 
 \]    
\end{definition}

The following are some results for composed reductions. 

\begin{lemma}\normalfont  \label{lemma:composition}
  Let $(A, \rightsquigarrow)$ be an abstract reduction system and 
  \[A^{\text{irr}} := \{w \in A \mid w \text{ is irreducible with respect to } \rightsquigarrow\}.\]
  Let $(A^{\text{irr}}, \rightarrowtail)$ be a second abstract reduction system and consider a new reduction given by the composition $\rightarrow \ : = \ \stackrel{*}{\rightarrowtail} \circ \stackrel{*}{\rightsquigarrow}$. Then $\rightarrow$ is transitive and reflexive. That is, 
  \[
\stackrel{*}\rightarrow
\ = \ \rightarrow.\]
\end{lemma}

\begin{proof}
Clearly, $\rightarrow \subseteq \stackrel{*}\rightarrow$. On the other hand, suppose $x \stackrel{*}\rightarrow z$.
If $x=z$, then $x \in A^{\text{irr}}$ and $x \stackrel{*}\rightarrow z$ implies that $x \stackrel{*}\rightarrowtail z$, hence $x \stackrel{*}\rightsquigarrow x \stackrel{*}\rightarrowtail  z$, so that $ x \rightarrow z$.
If $x \neq z$, then there exists $x_1 \in A^{\text{irr}}$ such that
\[
x \rightarrow x_1 \stackrel{*}\rightarrow z.
\]
Now, as $x \rightarrow x_1$, there exists $y \in A^{\text{irr}}$ such that
$x\stackrel{*}\rightsquigarrow y$ and $y \stackrel{*}\rightarrowtail x_1$. 
As $x_1 \in A^{\text{irr}}$, $x_1 \stackrel{*} \rightarrow z$ is the same as saying that
$x_1 \stackrel{*} \rightarrowtail z$ and it then follows that $y \stackrel{*}\rightarrowtail z$. Altogether, we have $y \in A^{\text{irr}}$ such that $x \stackrel{*} \rightsquigarrow y$ and $y \stackrel{*}\rightarrowtail z$, so that $x \rightarrow z$.
\end{proof}

\begin{proposition}  \normalfont \label{proposition:compositionARS}
 Let $(A, \rightsquigarrow)$ be an abstract reduction system with a locally confluent and terminating reduction $\rightsquigarrow$ and let $A^{\text{irr}} := {\{w \in A \mid w \text{ is irreducible with respect to } \rightsquigarrow\}}$. Suppose $(A^{\text{irr}}, \rightarrowtail)$ is an abstract reduction system where $\rightarrowtail$ is locally confluent and terminating. Let $\rightarrow \ : = \ \stackrel{*}{\rightarrowtail} \circ \stackrel{*}{\rightsquigarrow}$. Then every element of $A$ has a unique normal form with respect to the reduction $\rightarrow$.
\end{proposition}

\begin{proof}
Let $w \in A$. By Lemma \ref{lemma:uniquenormalform} and (Newman's) Lemma \ref{lemma:newman}, there exists unique $W(w) \in A^{\text{irr}}$ such that $w \stackrel{*}{\rightsquigarrow} W(w)$. Further, there exists unique $N(W
(w))\in A^{\text{irr}}$ such that $W(w)\stackrel{*}{\rightarrowtail} N(W(w))$ and $N(W(w))$ is irreducible with respect to the reduction $\rightarrowtail$. Hence $w \rightarrow N(W(w)$. We will show that $N(W(w))$ is the unique normal form of $w$ under the composed reduction $\rightarrow$. First, to show that $N(W(w))$ is irreducible under $\rightarrow$, suppose that $N(W(w)) \rightarrow z$. Then there exists $y \in A^{\text{irr}}$ such that $N(W(w)) \stackrel{*}{\rightsquigarrow} y \stackrel{*}{\rightarrowtail} z$. But $N(W(w)) \in A^{\text{irr}}$ implies that $y= N(W(w))$, so that we have $N(W(w)) \stackrel{*}{\rightarrowtail} z$. As $N(W(w))$ is also irreducible with respect to $\rightarrowtail$, we must have $z= N(W(w))$, so that $N(W(w))$ is irreducible and a normal form of $w$ with respect to the reduction $\rightarrow$. It remains to show that $N(W(w))$ is the \emph{unique} normal form of $w$ under $\rightarrow$. Suppose there exists another normal form $v \in A^{\text{irr}}$. Then $w \stackrel{*}{\rightarrow} v$, so by Lemma \ref{lemma:composition}, there exists $y \in A^{\text{irr}}$ with $w \stackrel{*}{\rightsquigarrow}  y \stackrel{*}{\rightarrowtail} v$. By the uniqueness of $W(w) \in A^{\text{irr}}$ such that $w \stackrel{*}{\rightsquigarrow} W(w)$, we must have $y=W(w)$, so that we are left with $W(w) \stackrel{*}{\rightarrowtail} v$ and $W(w) \stackrel{*}{\rightarrowtail} N(W(w))$. As $\rightarrowtail$ is confluent, there exists $u \in A^{\text{irr}}$ with $N(W(w)) \stackrel{*}{\rightarrowtail} u$ and $v \stackrel{*}{\rightarrowtail} u$. Now, as $N(W(w))$ is irreducible with respect to $\rightarrowtail$, we must have $u= N(W(w))$ and $v \stackrel{*}{\rightarrowtail} N(W(w))$. But this gives $v \stackrel{*}{\rightsquigarrow} v \stackrel{*}{\rightarrowtail} N(W(w))$, so that $v \rightarrow N(W(w))$. By the irreducibility of $v$ under $\rightarrow$, we must have $v=N(W(w))$.
\end{proof}

\subsection{The abstract reduction system for \texorpdfstring{$\cW(d)$}{Wd}} \label{subsection:ARSF}
We will use the relations of the Thompson group $F$ to formulate an abstract reduction system for $\cW(d)$. 
To simplify the use of the relations of $F$, we will actually work with two abstract reduction systems in sequence, and show that each of the reductions is terminating and locally confluent. We will then consider the composition of these two reductions and use Proposition \ref{proposition:compositionARS} to show a unique normal form exists with respect to the composed reduction.

The goal of our first abstract reduction system is to separate generators and inverses of generators from each other in a given word. In particular, we would like to move the generators, and the inverses of generators, to the left, and to the right respectively, of a word. The system is given by
$(\cW(d),\rightsquigarrow)$
with
\[u=w_1ww_2 \rightsquigarrow v = w_1f(w)w_2
\text{ for } u, v \in \cW(d),\] where $w$, a sub-word of length $2$, is rewritten as $f(w)$, another sub-word of length $2$, in the following way:

\begin{equation} \label{equation:rewriting1} 
f(w) = f(g_{\ui(k)}^{-1},g_{\ui(k+1)}) =
\begin{cases}
(g_{\ui(k+1)},g_{\ui(k)}^{-1}) & \ui(k)=\ui(k+1),\\
(g_{\ui(k+1)}, g_{\ui(k)+1}^{-1}) & \ui(k)>\ui(k+1),\\
(g_{\ui(k+1)+1}, g_{\ui(k)}^{-1}) & \ui(k+1)>\ui(k).\\
\end{cases} 
\end{equation}

Note that $\eval_F(u) = \eval_F(v)$ for $u, v \in \cW(d)$ whenever $u \rightsquigarrow v$.

\begin{proposition} \normalfont  \label{proposition:rewriting1}
The abstract reduction system $(\cW(d), \rightsquigarrow)$
is terminating and locally confluent.
\end{proposition}

\begin{proof}
The reduction system results in moving generators to the left of a given word. From \eqref{equation:rewriting1}, we observe that while the index of a generator (or its inverse) might change in a step of the reduction, generators remain generators (and inverses remain inverses). It is thus clear that for any word $w \in \cW(d)$, the reduction \eqref{equation:rewriting1} results in an irreducible word in at most $d^2$ steps (indeed, this is a very generous upper bound and the actual number of steps required to reach an irreducible word is less). Hence, there is no infinite chain $w_0 \rightsquigarrow w_1 \rightsquigarrow \cdots$, and the system $(\cW(d),\rightsquigarrow)$ is terminating. 

For local confluence, we need that if $u \rightsquigarrow v$ and $u \rightsquigarrow w$, then there exists $x\in \cW(d)$ such that $v \stackrel{*}{\rightsquigarrow} x$ and  $w \stackrel{*}{\rightsquigarrow} x$. 
By definition, $v$ and $w$ are obtained from $u$ by replacing some length $2$ sub-word $(g_{\ui(k)}^{-1}, g_{\ui(k+1)})$ by another length $2$ sub-word as in \eqref{equation:rewriting1}. It is easy to see that if two disjoint sub-words are rewritten in the reduction $u \rightsquigarrow v$ and $u \rightsquigarrow w$, then there exists $x$ as required, namely the word obtained by applying the two (commuting) reductions to $u$ one after the other.
It thus suffices to deal with the case of overlapping sub-words of length 2 in a word of length $3$. However, a quick check shows that a word $u$ of length $3$ cannot be reduced to $v$ and $w$ in two distinct ways under the relation $\rightsquigarrow$. Hence $(\cW(d),\rightsquigarrow)$ is locally confluent.
\end{proof}

As $(\cW(d),\rightsquigarrow)$ is locally confluent and terminating, it is confluent by Newman's Lemma \ref{lemma:newman}. Hence every word $w \in \cW(d)$ has a unique normal form. In particular, this means any word $w \in \cW(d)$ can be rewritten uniquely as 
\begin{equation} \label{equation:normal1} 
W(w) = (w_+,w_-),
\end{equation}
where $w_+$ is a sub-word consisting only of generators of $F$, and $w_-$ is a sub-word consisting only of inverses of generators. Observe that $\eval_F(w) = \eval_F(W(w))$. Let $\cW(d)^{\text{irr}}$ be the set of words of length $d$ which are given in the form $(w_+,w_-)$, i.e., it is the set of irreducible words of $\cW(d)$ under the reduction $\rightsquigarrow$.

\begin{example}\label{example:reduction1}
    Let $w = (g^{-1}_3,g_6,g^{-1}_0, g_2,g_4^{-1})$.
        Then we get
        \begin{figure}[H]
    \centering
    \begin{tikzcd}[column sep=tiny]
& (g_7,g^{-1}_3,g^{-1}_0,g_2,g_4^{-1}) \ar[dr, squiggly, dashed]
& \\
 (g^{-1}_3,g_6,g^{-1}_0, g_2,g_4^{-1}) \ar[ur, squiggly] \ar[dr, squiggly]
    &
      & {\color{blue} (g_7,g^{-1}_3,g_3,g^{-1}_0,g_4^{-1})  }\rightsquigarrow {\color{purple}(g_7,g_3,g_3^{-1},g_0^{-1},g_4^{-1})}
&  \\
&(g^{-1}_3,g_6,g_3,g_0^{-1},g_4^{-1}) \ar[ur, squiggly, dashed]
& 
\end{tikzcd}
\end{figure}
We observe here that applying the reduction $\rightsquigarrow$ to the disjoint sub-words $(g^{-1}_3,g_6)$ and $(g^{-1}_0,g_2)$ consecutively results in the same word, shown in blue, irrespective of the order in which the reductions are applied. The word shown in purple, $(g_7,g_3,g_3^{-1},g_0^{-1},g_4^{-1})$, is the unique normal form $W(w) \in \cW(d)^{\text{irr}}$.    
\end{example}

Our second abstract reduction system is given by $(\cW(d)^{\text{irr}}, \rightarrowtail)$ with  
\[u = w_1ww_2 \rightarrowtail v = w_1h(w)w_2,\] where $w$, a sub-word of length $2$, is rewritten as $h(w)$, another sub-word of length $2$, in the following way:

\begin{equation} \label{equation:rewriting2} 
h(w) = h(g_{\ui(k)}^{\uee(k)},g_{\ui(k+1)}^{\uee(k+1)}) =
\begin{cases}
(g_{\ui(k+1)-1}, g_{\ui(k)}) & \ui(k+1)-1 > \ui(k), \ \uee(k)=\uee(k+1)=1 \\
(g_{\ui(k+1)}^{-1}, g_{\ui(k)-1}^{-1}) & \ui(k)-1>\ui(k+1), \ \uee(k)=\uee(k+1)=-1\\
\end{cases}
\end{equation}
Note that $\eval_F(u)= \eval_F(v)$ for $u, v \in \cW(d)^{\text{irr}}$ whenever $u \rightarrowtail v$.
The goal of this reduction system is to start with a word $w=(w_+,w_-) \in \cW(d)^{\text{irr}}$ and rewrite $w_+$ so that the letters are arranged in decreasing order of their indices (except for a possible increase of $0$ or $1$), and symmetrically, to rewrite $w_-$ so that the letters are arranged in increasing order of their indices (except for a possible decrease of $0$ or $1$). This rewriting is aimed at arriving at the so -- called normal form \cite{DT19} (or by some authors, the anti-normal form \cite{Be04}) of an element of the Thompson monoid $F^+$. In the following proposition, we follow closely the argument in the proof of \cite[Lemma 2.1]{DT19}.

\begin{proposition} \normalfont \label{proposition:rewriting2}
The abstract reduction system $(\cW(d)^{\text{irr}}, \rightarrowtail)$,
is terminating and locally confluent.
\end{proposition}

\begin{proof}
Define $\tau: \cW(d)^{\text{irr}} \to \bN_0$ by
\begin{equation} \label{equation:sumindices}
\text{sum}(w) = \text{sum}(g_{\ui(1)}^{\uee(1)},\ldots, g_{\ui(d)}^{\uee(d)}) : = \sum_{k=1}^d \ui(k).
\end{equation}
Then from \eqref{equation:rewriting2}, it is clear that $w \rightarrowtail w' \implies \text{sum}(w) \geq \text{sum}(w')$, hence the reduction system is terminating by Lemma \ref{lemma:terminating}.

Next, we will show that $\rightarrowtail$ is locally confluent. In particular, we will show that if $u \rightarrowtail v$ and $u \rightarrowtail w$, then there exists $x\in \cW(d)^{\text{irr}}$ such that $v \stackrel{*}{\rightarrowtail} x$ and $w \stackrel{*}{\rightarrowtail} x$. 
By definition, $v$ and $w$ are obtained from $u$ by replacing some length $2$ sub-word $(g_{\ui(k)}^{\uee(k)}, g_{\ui(k+1)}^{\uee(k+1)})$ by another length $2$ sub-word. If the two reductions correspond to the rewriting of two disjoint sub-words, then the required $x$ is the word obtained by applying the two (commuting) reductions one after the other. So we only need to deal with the case of overlapping sub-words of length $2$ in a word of length $3$. This happens in the case that $u= (g_{\ui(k)},g_{\ui(k+1)},g_{\ui(k+2)})$ with $\ui(k+1)-1 \geq \ui(k)+1$, $\ui(k+2) -1\geq \ui(k+1)+1$. Then we have

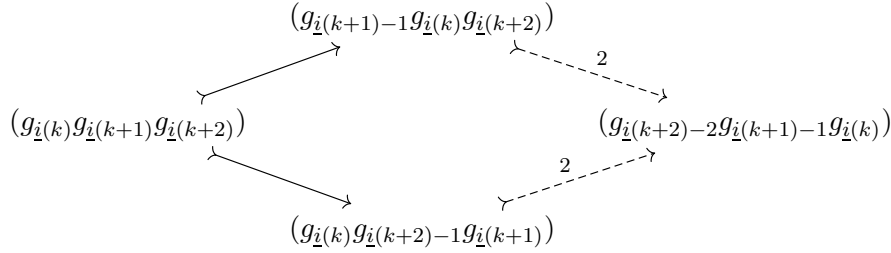
\begin{figure}[H]
    \centering
    \begin{tikzcd}[column sep=tiny]
& (g_{\ui(k+1)-1}g_{\ui(k)}g_{\ui(k+2)}) \ar[dr, tail, "2", dashed]
& \\
 (g_{\ui(k)}g_{\ui(k+1)}g_{\ui(k+2)}) \ar[ur, tail] \ar[dr, tail]
    &
      & (g_{\ui(k+2)-2}g_{\ui(k+1)-1}g_{\ui(k)})
&  \\
& (g_{\ui(k)}g_{\ui(k+2)-1}g_{\ui(k+1)})  \ar[ur, tail, "2", dashed]
& 
\end{tikzcd}
\end{figure}
 
A symmetric diamond can be drawn for inverses.
Hence $(\cW(d)^{\text{irr}},\rightarrowtail)$ is locally confluent, and consequently, it is confluent by Newman's Lemma \ref{lemma:newman}.
\end{proof}

As $(\cW(d)^{\text{irr}},\rightarrowtail)$ is confluent and terminating, every word $u$ has a unique normal form, which we write as $N(u)$. Moreover, $\eval_F(u) = \eval_F(N(u))$.

\begin{example}\label{example:reduction2}
Let us continue working with the word $W(w)$ that was obtained as the normal form in Example \ref{example:reduction1}. We get
\[
(g_7,g_3,g_3^{-1},g_0^{-1},g_4^{-1}) \rightarrowtail (g_7,g_3,g_0^{-1},g_2^{-1},g_4^{-1}),
\]
which is the required normal form $N(W(w)).$
\end{example}

Now, consider the abstract reduction system $(\cW(d), \rightarrow)$, where $\rightarrow \ := \ \stackrel{*}{\rightarrowtail} \circ \stackrel{*}{\rightsquigarrow}$.
By Proposition \ref{proposition:compositionARS}, as each of $(\cW(d), \rightsquigarrow)$ and $(\cW(d)^{\text{irr}}, \rightarrowtail)$ is terminating and locally confluent, the composed reduction system  $(\cW(d), \rightarrow)$ results in a unique normal form for each $w \in \cW(d)$.
Altogether, using Propositions \ref{proposition:rewriting1} and \ref{proposition:rewriting2}, we obtain for each word $w \in \cW(d)$, a unique normal form $N(W(w)) \in \cW(d)^{\text{irr}}$ as recorded in the proposition below. 

\begin{proposition} \normalfont 
 \label{prop:normalform}
Let $(\ui,\uee) \in \cW(d)$ be a word of length $d$. Then there exists a  unique word in $\cW(d)$ given by 
\begin{equation} \label{equation:normalform} (\uj, \uee') = (g_{\uj(1)},\cdots, g_{\uj(r)},g_{\uj(r+1)}^{-1},\cdots, g_{\uj(d)}^{-1})
\end{equation} satisfying the following conditions:
\begin{enumerate}
    \item $\eval_F(\ui,\uee) = \eval_F(\uj,\uee')$;
 \item $r\leq d$;
 \item $\uj(1) +1\geq \uj(2), \cdots, \uj(r-1) +1 \geq \uj(r)$;
 \item $\uj(r+1)\leq \uj(r+2)+1,\cdots, \uj(d-1) \leq \uj(d)+1$.
\end{enumerate}
\end{proposition}

 \begin{proof}
For $w = (\ui,\uee)$, let $(\uj,\uee')$ be defined by
\[
(\uj,\uee') := N(W(w)),
\]
where $W(w)$ is the unique normal form of $w$ obtained in Proposition \ref{proposition:rewriting1} under the reduction system $(\cW(d), \rightsquigarrow)$ and $N(W(w))$ is the unique normal form of $W(w)$ obtained in Proposition \ref{proposition:rewriting2} under the abstract reduction system $(\cW(d)^{\text{irr}}, \rightarrowtail)$.
 The form of $N(W(w))$ follows from \eqref{equation:rewriting1} and \eqref{equation:rewriting2}. By Proposition \ref{proposition:compositionARS}, $(\uj,\uee')$ is the unique normal form of the word $(\ui,\uee)$ under the reduction system $(\cW(d), \rightarrow)$.
 \end{proof}

\begin{remark} \normalfont \label{remark:interchangerules}
    Each step of the abstract reduction systems $(\cW(d), \rightsquigarrow)$ and $(\cW(d)^{\text{irr}},\rightarrowtail)$ results in two letters interchanging positions, causing
a change in the index of one letter by at most $ \pm 1$. Suppose that in a step of the reduction we have interchanged the positions of $g_{\ui(l)}^{\uee(l)}$ and $g_{\ui(m)}^{\uee(m)}$, with say, $\ui(l) > \ui(m)$. Then
the change in index can be summarized as follows. 

\begin{itemize}
\item[--] If $g_{\ui(l)}$ moves to the \emph{left} of $g_{\ui(m)}^{-1}$, 
then
$g_{\ui(l)}$ is transformed to $g_{\ui(l)+1}$.

\item[--] If $g_{\ui(l)}^{-1}$ moves to the \emph{right} of $g_{\ui(m)}$, 
then
$g_{\ui(l)}^{-1}$ is transformed to $g_{\ui(l)+1}^{-1}$.

\item[--] If $g_{\ui(l)}$ moves to the \emph{left} of $g_{\ui(m)}$ with $\ui(l) -1 > \ui(m)$, 
then $g_{\ui(l)}$ is transformed to $g_{\ui(l)-1}$.

\item[--] If $g_{\ui(l)}^{-1}$ moves to the \emph{right} of $g_{\ui(m)}^{-1}$ with $\ui(l) -1 > \ui(m)$, 
then $g_{\ui(l)}^{-1}$ is transformed to $g_{\ui(l)-1}^{-1}$.
\end{itemize}
We reiterate that while the indices of letters might change in each step of the reduction, generators remain generators (and inverses remain inverses). Moreover, the letter with the smaller index $g_{\ui(m)}^{\uee(m)}$ remains unchanged.
\end{remark}

\section{A method for putting neutral words into bins} \label{section:algorithm}
In the previous section, we described an abstract reduction system $(\cW(d), \rightarrow)$ and showed that each word in $\cW(d)$ has a unique normal form. In this section, we will use the structure of this normal form to assign a pair-partition -- a ``bin'' -- to each neutral word of $\cW(d)$.
We start this section with some notations and definitions reviewing permutation groups and pair-partitions.

\begin{notation}[Review of notations about permutations]
$\cS_d$ will denote the group of permutations of the set $[d]$ for $d \in \bN$.  
\end{notation}

\begin{notation}[Review of notations about pair-partitions]
$\cP_2(d)$ will denote the set of pair-partitions of the set $[d]$ for $d \in \bN$. $\cP_2(d)$ is non-empty only for $d$ even, and in that case, a pair-partition is of the form $\pi = \{V_1, \ldots, V_{\frac{d}{2}}\}$ with $V_i \subseteq [d]$ and $|V_i|=2$ for each $i \in [\frac{d}{2}]$.
\end{notation}

\begin{notation-and-definition}\label{notdef:rainbow}
For $d$ an even integer, $\piR \in \cP_2(d)$ will denote the rainbow pair-partition given by
\[
\piR = \{\{1, d\}, \ldots, \{\frac{d}{2}, \frac{(d+2)}{2}\}\}.
\]
\end{notation-and-definition}

We will use the abstract reduction system $(\cW(d), \rightarrow)$ discussed in Subsection \ref{subsection:ARSF} to associate to each neutral word in $\cW(d)$ a unique pair partition of the set $[d]$. The pair partitions of $[d]$ will be the ``bins'' into which we place our words. We will do this by first assigning to each word a permutation of $[d]$, and then assigning a pair partition to the permutation using the rainbow pair-partition as a tool. Let us now describe how to find a permutation for a given word using its normal form described by \eqref{equation:normalform}.

\subsection{Assigning permutations to words}
\label{subsection:PermutationsGivenNeutralWords} 
The abstract reduction system $(\cW(d), \rightarrow)$ results in a unique normal form \eqref{equation:normalform} for each word $w =(\ui,\uee) \in W(d)$. For each $l \in [d]$, $g_{\ui(l)}^{\uee(l)}$ is transformed uniquely to some $g_{\uj(u_l)}^{\uee(l)}$ in the normal form $(\uj,\uee')$. The steps in the transformation of the index $\ui(l)$ to $\uj(u_l)$ is described by Remark \ref{remark:interchangerules}. Define the permutation $\tau(\ui,\uee) \in \cS_d$ by
 
\begin{equation}\label{equation:permdef}
\tau(\ui,\uee)(l) := u_l, \ l\in [d].
\end{equation}

\begin{lemma}\normalfont  \label{lemma:perm}
\normalfont
Let $w = (\ui, \uee) \in \cW(d)$ and $(\uj, \uee')$ be its unique normal form as in \eqref{equation:normalform}. Let $\tau(\ui,\uee) \in \cS_d$ be the permutation defined as in \eqref{equation:permdef}. Then we have the following inequalities:
\[-d \leq \uj(l) - \ui(\tau(\ui,\uee)^{-1}(l)) \leq  d, \quad l \in [d].\] 
 If $\sum_{l=1}^d \uee(l) =0$, then $d$ is even, and 
 \[
 \uee' = \uee_0 := (\underbrace{+1, \ldots, +1}_{\frac{d}{2} \text{ times }}, \underbrace{-1, \ldots, -1}_{\frac{d}{2} \text{ times }})
 \]
Further, in this case, we have the inequalities:
\[-\frac{d}{2} \leq \uj(l) - \ui(\tau(\ui,\uee)^{-1}(l)) \leq  \frac{d}{2}, \quad l \in [d].\] 
\end{lemma}

\begin{proof}
The reduction of a word $(\ui,\uee)$ of length $d$ to its normal form $(\uj,\uee')$ can result in at most $d$ position changes for a given letter, and correspondingly can cause a change of at most $\pm d$ in its index, as detailed in Remark \ref{remark:interchangerules}. As $\uj(l)$ is the transformed index of a letter which was originally in the position $\tau(\ui,\uee)^{-1}(l)$ in the original word $(\ui,\uee)$, we have
\[
-d \leq \uj(l) - \ui(\tau(\ui,\uee)^{-1}(l)) \leq d, \ l \in [d]
\]

Next, suppose that $\sum_{t=1}^d \uee(t) =0$. As $\uee$ takes values in $\{-1, 1\}$, this means that $d$ must be even and that the given word is composed of exactly $\frac{d}{2}$ generators and $\frac{d}{2}$ inverses of generators. Hence $\uee' = \uee_0$, where
\[
\uee_0 := (\underbrace{+1, \ldots, +1}_{\frac{d}{2} \text{ times }}, \underbrace{-1, \ldots, -1}_{\frac{d}{2} \text{ times }})
\]
Now,
the abstract reduction system only allows for any letter to move to the right (or to the left) of at most $\frac{d}{2}$ generators, and at most $\frac{d}{2}$ inverses of generators. Hence by Remark \ref{remark:interchangerules}, any index $\ui(\tau(\ui,\uee)^{-1}(l))$ can change by at most $\pm \frac{d}{2}$ to some $\uj(l)$ in the normal form $(\uj,\uee_0)$.
\end{proof}

Let $d, n\in \bN$. Consider the following sets of words
\begin{equation*} 
\cW_0(d) := \Bigl\{  ( \ui , \uee )  \begin{array}{ll}
\vline & \ui : [d] \to \bN, \ \uee : [d] \to \{ -1,1 \}, \\
\vline & \mbox{such that} \ \eval_F(\ui,\uee) = g_{\ui (1)}^{\uee (1)} \cdots g_{\ui (d)}^{\uee (d)} = e 
\end{array}  \Bigr\},
\end{equation*}

and 
\begin{equation*} 
\neutral := \{(\ui,\uee) \in \cW_0(d) \mid \ui: [d] \to \{0,\ldots, n-1\}\}.
\end{equation*}

In other words, $\cW_0(d)$ is the set of neutral words in $\cW(d)$ and $\neutral$ is the set of neutral words of length $d$ composed of the first $n$ generators $\{g_0, \ldots, g_{n-1}\}$ of $F$. It is immediate that $\cW_0(d)$ and $\neutral$ are non-empty if and only if $d$ is even.
We record some easy corollaries of Lemma \ref{lemma:perm}. 

\begin{corollary} \normalfont \label{corollary:normalform}
Let $(\ui, \uee) \in \cW_0(d)$. Then the normal form of $(\ui,\uee)$ is given by

\begin{equation} \label{equation:neutralnormalform}
(\uj,\uee_0) = (g_{\uj(1)},\ldots g_{\uj(\frac{d}{2})},g_{\uj(\frac{d}{2})}^{-1},\ldots ,g_{\uj(1)}^{-1}) 
\end{equation}
with $\uj(1)+1\geq \uj(2), \ldots, \uj(\frac{d}{2}-1)+1 \geq \uj(\frac{d}{2})$.
Further, the permutation $\tau(\ui,\uee) \in \cS_d$ satisfies
\begin{equation} \label{equation:ij-inequality}
-\frac{d}{2} \leq \uj(l) -\ui(\tau(\ui,\uee)^{-1}(l)) \leq \frac{d}{2}, \ l \in [\frac{d}{2}].
\end{equation}
\end{corollary}

\begin{corollary}\normalfont  \label{corollary:ineq}
Let $(\ui, \uee) \in \cW_0(d)$ and $\tau(\ui,\uee) =\tau$. Then for each $k \in [\frac{d}{2}]$, 
\[
|\ui(\tau^{-1}(k)) - \ui(\tau^{-1}(d-k+1))| \leq d.
\]
\end{corollary}

\begin{proof}
    From Corollary \ref{corollary:normalform}, we have that given a word $(\ui,\uee) \in \cW_0(d)$ and its unique normal form $(\uj,\uee_0)$,
\[
\frac{-d}{2} \leq \uj((l) - \ui(\tau^{-1}((l)) \leq \frac{d}{2}, \ l \in [d].
\]
As $\uj(k) = \uj(d-k+1)$ for each $k \in [\frac{d}{2}]$, we arrive at
\begin{align*}
        |\ui(\tau^{-1}(k)) - \ui(\tau^{-1}(d-k+1)| 
 & \leq |\ui(\tau^{-1}(k)) - \uj(k)|+ |\uj(d-k+1) - \ui(\tau^{-1}(d-k+1))|     \\
 & \leq \frac{d}{2}+\frac{d}{2} = d.
\end{align*}
\end{proof}

\subsection{Assigning pair partitions to neutral words} \label{subsection:NeutralWordsGivePairPartitions}

We have already assigned to each word $(\ui,\uee)$ in $\cW(d)$ (and in particular, in $\cW_0(d)$) a permutation $\tau(\ui,\uee)$ corresponding to the normal form of $(\ui,\uee)$. 
The structure of the normal form $(\uj,\uee_0)$ in \eqref{equation:neutralnormalform} makes it evident why $\eval_F(\ui,\uee) =e$. Indeed, successive generators $g_{\uj(\frac{d}{2})}$, $g_{\uj{(\frac{d}{2}-1})}, \ldots, g_{\uj(1)}$ and their inverses clearly cancel each other out in $\eval_F(\uj,\uee_0)$. We use this observation to associate a pair partition of the set $[d]$ to the word $(\ui,\uee)$.

Let $\piR$ denote the rainbow pair-partition as in Notation and Definition \ref{notdef:rainbow}. For any permutation $\sigma \in \cS_d$ we write  $\sigma^{-1}(\piR)$ to mean the pair partition
\[
\{\{\sigma^{-1}(1),\sigma^{-1}(d)\}, \{\sigma^{-1}(2),\sigma^{-1}(d-1)\}, \ldots, \{\sigma^{-1}(\frac{d}{2}),\sigma^{-1}(\frac{d+2}{2})\}\}.
\]

Define the map $\pi$ from $\cW_0(d) \to \cP_2(d)$ as: 

\begin{equation} \label{equation:permrainbow}
\pi(\ui,\uee) :=  (\tau(\ui, \uee))^{-1} (\piR).  
\end{equation}
Hence, for each word $(\ui,\uee) \in \cW_0(d)$, we obtain a unique ``bin'' or pair partition, given by $\pi(\ui,\uee)$. This pair partition records which generator and inverse generator from the original neutral word $(\ui,\uee)$ eventually ``pair up'' in the normal form $(\uj,\uee_0)$ to cancel each other out in the evaluation of the word as the identity element.
Indeed, as the rainbow pair partition $\piR$ consists of pairs of the form $\{k,d-k+1\}$ for $k \in [\frac{d}{2}]$, the pair partition $\pi(\ui,\uee)$ consists of pairs of the form $\{\tau(\ui,\uee)^{-1}(k), \tau^{-1}(d-k+1)\}$ for $k \in [\frac{d}{2}]$.
We illustrate with an example. 

\begin{example}\label{example:illustrative}
Suppose $d=8$ and $(\ui,\uee) \in \cW_0(8)$ is given by
\[(\ui,\uee) = (g_{2},g_0,g_{18}^{-1},g_4^{-1},g_0^{-1},g_{16},g_3,g_2^{-1}).\]
Then the normal form $(\uj,\uee_0)$ of $(\ui,\uee)$ is
\[
(\uj,\uee_0) = (g_{16},g_2,g_3,g_0,g_0^{-1},g_3^{-1}, g_2^{-1}, g_{16}^{-1}).
\]
We visualize the rainbow pair partition on the normal form as
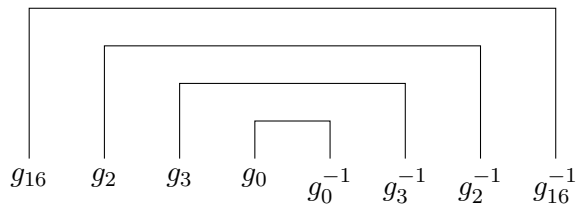
\[ \begin{tikzpicture}[scale=0.5]
    \draw (0,0) node[below] {$g_{16}$} -- (0,4) --  (14,4) -- (14,0)  node[below] {$g_{16}^{-1}$};
    \draw (2,0) node[below] {$g_2$} -- (2,3) -- (12,3) -- (12,0) node[below] {$g_{2}^{-1}$};
    \draw (4,0) node[below] {$g_{3}$} -- (4,2) -- (10,2) -- (10,0) node[below] {$g_{3}^{-1}$};
    \draw (6,0) node[below] {$g_{0}$} -- (6,1) -- (8,1) -- (8,0) node[below] {$g_{0}^{-1}$};
    \end{tikzpicture} \]
The permutation $\tau(\ui,\uee)$ is given by
\[ \tau(\ui,\uee) = (1,2,4,6)(3,8,7).
\]
For instance, we see that the letter $g_{\ui(1)}=g_2$ is transformed to $g_{\uj(2)}=g_2$ in $(\uj,\uee_0)$ (so $\tau(\ui,\uee)(1)=2$), or the letter $g_{\ui(3)}^{-1}= g_{18}^{-1}$ is transformed to $g_{\uj(8)}^{-1}=g_{16}^{-1}$ in $(\uj,\uee_0)$ (so $\tau(\ui,\uee)(3)=8$). We can also verify as given in Lemma \ref{lemma:perm}, for instance, that
\[
|\uj(8)- \ui(\tau(\ui,\uee)^{-1}(8))| = |\uj(8)-\ui(3)| = |16-18| =2 \leq 4=\frac{8}{2} =\frac{d}{2}.
\]

The corresponding pair partition
is then given by 
\[\pi(\ui,\uee) = \left\{\{1,8\}, \{2,5\}, \{3,6\}, \{4,7\}\right\},\]
The pair partition $\pi(\ui,\uee)$ is visualized below to indicate which generators and which inverses pair up in the evaluation of the normal form to cancel each other out.
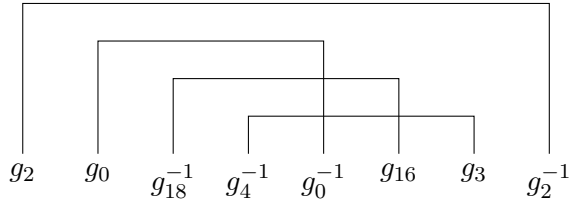
\[ \begin{tikzpicture}[scale=0.5]
    \draw (0,0) node[below] {$g_{2}$} -- (0,4) --  (14,4) -- (14,0) node[below] {$g_{2}^{-1}$} ;
     \draw (2,0) node[below] {$g_{0}$} -- (2,3) -- (8,3) -- (8,0) node[below] {$g_{0}^{-1}$};
     \draw (4,0) node[below] {$g_{18}^{-1}$} -- (4,2) -- (10,2)  -- (10,0) node[below] {$g_{16}$} ;
    \draw (6,0) node[below] {$g_{4}^{-1}$} -- (6,1) -- (12,1) -- (12,0) node[below] {$g_{3}$};
    \end{tikzpicture} \]
 \end{example}

\subsection{Obtaining neutral words from permutations and pair partitions} \label{subsection:PairPartitionsGiveNeutralWords}
In Subsection \ref{subsection:NeutralWordsGivePairPartitions}, we described how to assign a unique permutation and a pair partition to every neutral word. In this subsection, we describe some constructions in the other direction. That is, given a permutation (or in some cases, only a pair partition) we can complete a partially filled word to give a neutral word which corresponds to the given permutation (or pair partition, respectively). 
We start with a simple but key lemma.

\begin{lemma} \normalfont 
 \label{lemma:smallestindex}
Let $(\ui,\uee) \in \cW_0(d)$ with $i_0 = \min \{\ui(l) \mid l \in [d]\}$. For each pair $\{l,m\} \in \pi(\ui,\uee)$, $\ui(l) = i_0$ if and only if $\ui(m) = i_0$.   
\end{lemma}

\begin{proof}
Remark \ref{remark:interchangerules} explains that the generator with smaller index is unchanged in any sub-word of length $2$ when the reduction $\rightarrow$ is applied to a word. Hence, $\ui(k) = i_0$ if and only if $\uj(\tau(\ui,\uee)(k))=i_0$. Given the pair $\{l,m\} \in \pi(\ui,\uee)$, $\ui(l)=i_0$ if and only if $\ui(m) = \uj(\tau(\ui,\uee)(m)) = \uj(\tau(\ui,\uee)(l)) = \ui(l) = i_0$.
\end{proof}

In the following proposition, we will use Lemma \ref{lemma:smallestindex} to show that a neutral word $(\ui,\uee)$ can be determined by $\pi(\ui,\uee)$ if $\uee$ is completely known, and $\ui$ is partially known. We use a new abstract reduction system with the goal of pushing out the generators with the smallest index to the left of a word, and the inverse generators with the smallest index to the right of a word. The reason we use this new reduction system rather than $(\cW(d), \rightarrow)$ used earlier is that we cannot be guaranteed that the indices of the letters provided are sufficiently spread out. Compare this with Proposition \ref{proposition:spaceddetermineneutral} where the indices are spread out enough for us to know precisely how the reduction $\rightarrow$ changes an index, even when the actual value of the index is unknown.

\begin{proposition} \normalfont \label{proposition:determineneutral}
 Let $d$ be an even integer. Suppose that $(\ui,\uee) \in \cW_0(d)$, and the tuple $\uee : [d] \to \{-1,1\}$ is given such that
 $\uee(l_+) = 1$ and $\uee(l_-) = -1$ for each pair $\{l_+,l_-\} \in \pi(\ui,\uee)$. If the values of $\ui(l_+)$ are known for each $\{l_+,l_-\} \in \pi(\ui,\uee)$, then the values of $\ui(l_-)$ are uniquely determined.
  \end{proposition}

\begin{proof}
We prove the proposition by induction on $d$, the length of the word. Suppose that $d=2$. The pair partition must be given by $\pi(\ui,\uee) = \{\{1,2\}\}$. Then we must have $\ui(l_-)=\ui(l_+)$ by Lemma \ref{lemma:smallestindex}. 

Next, suppose the proposition holds for all (even) $d<d_0$. Suppose $(\ui,\uee) \in \cW_0(d_0)$ with the value of $\ui(l_+)$ given for each $\{l_+,l_-\}\in \pi(\ui,\uee)$. Let $i_0= \min\{\ui(l) \mid l \in [d]\}$. Then by Lemma \ref{lemma:smallestindex}, we must have $\ui(l_-)=i_0$ whenever $\ui(l_+)= i_0$.
To show that the remaining values of $\ui$ can be uniquely determined, we will use a new abstract reduction system $(\cW(d),\twoheadrightarrow)$. The goal is to push outwards those letters whose indices are the smallest for a given word, and thereby arrive at a shorter word on which the induction hypothesis can be used. The reduction $\twoheadrightarrow$ is given by
\[
u = w_1ww_2 \twoheadrightarrow  v= w_1f(w)w_2.
\]

Here $w$, a sub-word of length $2$, is rewritten as $f(w)$ with the following rules:

\begin{align*}
f(g_{\ui(m)}^{\uee(m)}, g_{i_0}) &=  
    (g_{i_0},g_{\ui(m)+1}^{\uee(m)}), & \text{ if } i_0 = \min\{\ui(l) \mid l \in [d]\} \text{ and } \ui(m) > i_0 \\
f(g_{i_0}^{-1},g_{\ui(m)}^{\uee(m)}) &=  
(g_{\ui(m)+1}^{\uee(m)},g_{i_0}^{-1}), & \text{ if } i_0 = \min\{\ui(l) \mid l \in [d]\} \text{ and } \ui(m) > i_0 \\
f(g_{i_0}^{-1},g_{i_0}) & = (g_{i_0},g_{i_0}^{-1}) &   \text{ if } i_0 = \min\{\ui(l) \mid l \in [d]\}. 
\end{align*}

We will show that $\twoheadrightarrow$ is locally confluent and terminating. It is terminating because in a given word $(\ui,\uee)$ with $\min\{\ui(l) \mid l \in [d]\} = i_0$, there are at most $d$ occurrences of $g_{i_0}$ (and $g_{i_0}^{-1}$) and each can be moved at most $(d-1)$ times so that the rewriting terminates in less than $2d(d-1)$ steps.
For local confluence, consider the following diamond similar to the proof of Proposition \ref{proposition:rewriting2}:
\begin{figure}[H]
    \centering
    \begin{tikzcd}[column sep=tiny]
& g_{\ui(m)+1}^{\uee(m)}g_{i_0}^{-1}g_{i_0} \ar[dr, two heads, "2", dashed]
& \\
 g_{i_0}^{-1}g_{\ui(m)}^{\uee(m)}g_{i_0} \ar[ur, two heads] \ar[dr, two heads]
    &
      & g_{i_0}g_{\ui(m)+2}^{\uee(m)}g_{i_0}^{-1}
&  \\
& g_{i_0}^{-1}g_{i_0}g_{\ui(m)+1}^{\uee(m)}  \ar[ur, two heads, "2", dashed]
& 
\end{tikzcd}
\end{figure}
Hence, Newman's lemma gives a unique normal form of the following type:
\[
g_{i_0}\cdots g_{i_0} \tilde{w} g_{i_0}^{-1} \cdots g_{i_0}^{-1}.
\]
Here $\tilde{w} =(\tilde{\ui},\tilde{\uee})$ is a uniquely determined word of length strictly less than $d_0$, where $\tilde{\uee}$ is known, and the values of $\tilde{\ui}$ are known in terms of the values of $\ui$.
Further, $\eval_F(\tilde{w})=e$, and the pairs in $\pi(\tilde{\ui},\tilde{\uee})$ correspond to the pairs in $\pi(\ui,\uee)$. By the induction hypothesis, $\tilde{\ui}$ can be determined completely. We have already determined $\ui(l_-) = i_0$ for all pairs $\{l_+,l_-\} \in \pi(\ui,\uee)$ with $\ui(l_+)=i_0$. For the remaining pairs $\{m_+,m_-\} \in \pi(\ui,\uee)$, the indices $\ui(m_-)$ can now be recovered from the tuple $\tilde{\ui}$.
\end{proof}

We demonstrate with an example.
\begin{example}
Suppose $(\ui,\uee) \in W_0(6)$ with pair partition $ \pi(\ui,\uee) = \{\{1,4\}, \{2,5\}, \{3,6\}\}$ and $\uee=(1,-1,-1,-1,1,1)$. The value of $\ui(l_+)$ is given for each $\{l_+,l_-\} \in \pi(\ui,\uee)$:
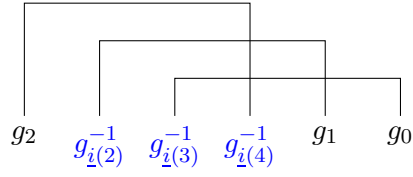
    \[ \begin{tikzpicture}[scale=0.5]
         \draw (0,0) node[below] {$g_{2}$} -- (0,3) -- (6,3) -- (6,0)  node[below] {\color{blue} $g_{\ui(4)}^{-1}$} ;
     \draw (2,0)  node[below] {\color{blue} $g_{\ui(2)}^{-1}$} -- (2,2) -- (8,2)  -- (8,0) node[below] {$g_{1}$} ;
    \draw (4,0)   node[below] {\color{blue} $g_{\ui(3)}^{-1}$} -- (4,1) -- (10,1) -- (10,0) node[below] {$g_{0}$};
    \end{tikzpicture} \]
We will fill in the missing inverse generators by using Proposition \ref{proposition:determineneutral}.
We identify the smallest index $i_0=0 = \ui(6)$ and fill in $\ui(3)=\ui(6)=0$ as $\{3,6\} \in \pi(\ui,\uee)$.
The unique normal form under the reduction $\twoheadrightarrow$ is
\[(g_{0},g_3,g_{\ui(2)+1}^{-1}, g_{\ui(4)+2}^{-1},g_3,g_0^{-1}),\]

The shorter word is given by $\tilde{w} = (g_3,g_{\ui(2)+1}^{-1},g_{\ui(4)+2}^{-1},g_3)$ and shown below with the relevant pair partition:
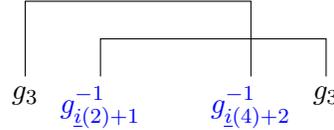
\[ \begin{tikzpicture}[scale=0.5]
         \draw (0,0) node[below] {$g_{3}$} -- (0,2) -- (6,2) -- (6,0)  node[below] {\color{blue} $g_{\ui(4)+2}^{-1}$} ;
     \draw (2,0)  node[below] {\color{blue} $g_{\ui(2)+1}^{-1}$} -- (2,1) -- (8,1)  -- (8,0) node[below] {$g_{3}$} ;
       \end{tikzpicture} \]

Now identifying the smallest index in $\tilde{w}$ as $3$, we get $\ui(4)+2=3$ and $\ui(2)+1=3$, so the completed word is as follows:
        
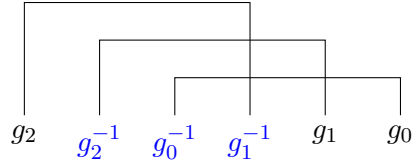
\[ \begin{tikzpicture}[scale=0.5]
         \draw (0,0) node[below] {$g_{2}$} -- (0,3) -- (6,3) -- (6,0) node[below] {\color{blue} $g_{1}^{-1}$};
     \draw (2,0) node[below] {\color{blue}$g_{2}^{-1}$} -- (2,2) -- (8,2)  -- (8,0) node[below] {$g_{1}$} ;
    \draw (4,0) node[below] {\color{blue}$g_{0}^{-1}$} -- (4,1) -- (10,1) -- (10,0) node[below] {$g_{0}$};
    \end{tikzpicture} \]
\end{example}

Above, we demonstrated how a partially-filled word $(\ui,\uee)$ that is known to be neutral can be completed uniquely given its associated pair partition $\pi(\ui,\uee)$. We will use this in Proposition \ref{proposition:upperbound}.

Next, we will show that given a permutation $\tau$, and letters with sufficiently spaced out indices, we can always fill in inverse generators uniquely to arrive at a neutral word $(\ui,\uee)$ such that $\tau(\ui,\uee) =\tau$. Here, we will use our original reduction system $(\cW(d), \rightarrow)$.

\begin{proposition} \normalfont \label{proposition:spaceddetermineneutral}
 Let $d$ be an even integer and $\tau \in \cS_d$ be a permutation. 
 Let the tuple $\uee : [d] \to \{-1,1\}$ be given such that
\[
\uee(\tau^{-1}(l)) = \begin{cases}
    1 & \text{ if } l \in \{1,\ldots, \frac{d}{2}\} \\
    -1 & \text{ if } l \in \{\frac{d}{2}+1, \ldots, d\}.
\end{cases}
\]
Suppose we are given values 
$\ui(\tau^{-1}(l))$ for $l \in \{1,\ldots, \frac{d}{2}\}$ such that
 \begin{equation} \label{equation:spaced}
 \ui(\tau^{-1}(l)) < \ui(\tau^{-1}(l-1)) -3d, \ l \in \{2,\ldots, \frac{d}{2}\}.
 \end{equation}
Then there exist unique values $\ui(\tau^{-1}(l))$ for $l \in \{\frac{d}{2}+1, \ldots, d\}$ such that $(\ui,\uee) \in \cW_0(d)$ and $\tau(\ui,\uee)= \tau$.
  \end{proposition}

\begin{proof}
By Corollary \ref{corollary:ineq}, in order to ensure that $(\ui,\uee) \in \cW_0(d)$ with $\tau(\ui,\uee)=\tau$, a necessary condition is
 \[
 |\ui(\tau^{-1}(k))-\ui(\tau^{-1}(d-k+1))| \leq d, \ k \in [\frac{d}{2}].
\]
Hence we will restrict our choice of $\ui(\tau^{-1}(d-k+1))$ to satisfy

\begin{equation} \label{equation:restriction}
\ui(\tau^{-1}(k))-d \leq \ui(\tau^{-1}(d-k+1)) 
 \leq \ui(\tau^{-1}(k))+d,  \ k \in [\frac{d}{2}]. \end{equation}

Recall the abstract reduction system $(\cW(d), \rightarrow)$ described in Section \ref{subsection:ARSF} where $\rightarrow$ is the composed reduction given by $\rightarrow \, = \, \stackrel{*}{\rightarrowtail} \circ \stackrel{*}\rightsquigarrow$. Let $(\uj,\uee_0)$ be the unique normal form of $(\ui,\uee)$ under $(\cW(d), \rightarrow)$.
As described in the beginning of Subsection \ref{subsection:PermutationsGivenNeutralWords} and in the proof of Lemma \ref{lemma:perm}, each letter $g_{\ui(\tau^{-1}(l))}^{\uee(\tau^{-1}(l))}$ is transformed to a letter $g_{\uj(v_l)}^{\uee(\tau^{-1}(l))}$, where 
\begin{equation} \label{equation:transformed}
   v_l = 
 \tau(\ui,\uee)(\tau^{-1}(l)), \
 \uj(v_l)= \ui(\tau^{-1}(l)) +q(l), \ \text{and } q(l) \in \{-\frac{d}{2}, \ldots, 0, \ldots, \frac{d}{2}\}
\end{equation}
 We claim that the values of $q(l)$ for each $l \in [\frac{d}{2}]$ can be explicitly computed due to the inequalities \eqref{equation:spaced} and \eqref{equation:restriction}. Indeed, in each step of the first reduction $\rightsquigarrow$, we will rewrite a sub-word of length $2$ of the form:
 \[
 (g^{-1}_{\ui(\tau^{-1}(d-m+1))+p}, g_{\ui(\tau^{-1}(k))+q}), \ k, m \in [\frac{d}{2}], -d \leq p,q \leq d,
 \]
 where $p,q$ are known quantities (for example, in the very first step, $p=q=0$).
 If $m>k$,
 \[\ui(\tau^{-1}(d-m+1)) +p \leq \ui(\tau^{-1}(m)) +d+p < \ui(\tau^{-1}(k))-2d +p \leq \ui(\tau^{-1}(k))-d \leq \ui(\tau^{-1}(k)) +q .\]
 Hence the rewriting \eqref{equation:rewriting1} is given by
 \[
 (g^{-1}_{\ui(\tau^{-1}(d-m+1))+p}, g_{\ui(\tau^{-1}(k))+q}) \rightsquigarrow (g_{\ui(\tau^{-1}(k))+q+1}, g^{-1}_{\ui(\tau^{-1}(d-m+1))+p}).
  \]
Similarly, if $m<k$ or if $m=k$, the 
rewriting \eqref{equation:rewriting1} can be writted precisely. We summarize as follows:
 \begin{align*}
(g^{-1}_{\ui(\tau^{-1}(d-m+1))+p}, g_{\ui(\tau^{-1}(k))+q}) & \rightsquigarrow (g_{\ui(\tau^{-1}(k))+q+1}, g^{-1}_{\ui(\tau^{-1}(d-m+1))+p}) & m>k,\\
(g^{-1}_{\ui(\tau^{-1}(d-m+1))+p}, g_{\ui(\tau^{-1}(k))+q}) & \rightsquigarrow (g_{\ui(\tau^{-1}(k))+q}, g^{-1}_{\ui(\tau^{-1}(d-m+1))+p+1}) & m < k,\\
(g^{-1}_{\ui(\tau^{-1}(d-m+1))+p}, g_{\ui(\tau^{-1}(k))+q}) & \rightsquigarrow (g_{\ui(\tau^{-1}(k))+q}, g^{-1}_{\ui(\tau^{-1}(d-m+1))+p}) & m=k.
  \end{align*}
  
The second reduction $\rightarrowtail$ gives rewritings of sub-words of length $2$ of the form:
 \begin{align*}
 (g_{\ui(\tau^{-1}(m))+p}, g_{\ui(\tau^{-1}(k))+q}) & \rightarrowtail (g_{\ui(\tau^{-1}(k))+q-1}, g_{\ui(\tau^{-1}(m))+p}) & k<m, \\
 (g^{-1}_{\ui(\tau^{-1}(d-m+1))+p}, g^{-1}_{\ui(\tau^{-1}(d-k+1))+q}) & \rightarrowtail  (g^{-1}_{\ui(\tau^{-1}(d-k+1))+q}, g^{-1}_{\ui(\tau^{-1}(d-m+1))+p-1})  & m<k.
 \end{align*}

Now on applying the composed reduction $\rightarrow \, = \, \stackrel{*}{\rightarrowtail} \circ \stackrel{*}{\rightsquigarrow}$ to $(\ui,\uee)$, we arrive at the normal form
$(\uj,\uee_0)$ with the value of each $q(l)$ known in \eqref{equation:transformed}. We remind that by definition of the permutation $\tau(\ui,\uee)$, each generator $g_{\ui(\tau(\ui,\uee)^{-1}(l))}$ is transformed to the generator $g_{\uj(l)}$. We also remind that the normal form $(\uj,\uee_0)$ satisfies
\begin{equation} \label{equation:normalformrep}
  \uj(l) \leq \uj(l-1)+1, \ l \in \{2,\ldots, \frac{d}{2}\}.  
\end{equation}

By the hypothesis \eqref{equation:spaced}, we have
\[
\ui(\tau^{-1}(l)) <  \ui(\tau^{-1}(l-1)) -3d, \ l \in \{2, \ldots, \frac{d}{2}\}.
\]

Hence, for any choice of $q(l-1), q(l) \in \{-\frac{d}{2}, \ldots, \frac{d}{2}\}$, we have

\begin{align*}
\ui(\tau^{-1}(l)) +q(l) & < \ui(\tau^{-1}(l-1)) +q(l)-3d \\
&\leq \ui(\tau^{-1}(l-1)) +\frac{d}{2}-3d \\
&= \ui(\tau^{-1}(l-1)) -\frac{5d}{2} \\
&\leq \ui(\tau^{-1}(l-1)) +q(l-1) \\
&\leq \ui(\tau^{-1}(l-1)) +q(l-1)+1.
\end{align*}
Thus we get
\[\uj(v_{l}) \leq \uj(v_{l-1}) +1, \ l \in \{2,\ldots, \frac{d}{2}\}.\]

By the uniqueness of the normal form, and comparing with \eqref{equation:normalformrep}, we must have $l=v_l$ for every $l \in [\frac{d}{2}]$. A symmetric argument gives that $l=v_l$ for every $l \in\{\frac{d}{2}+1, \ldots, d\}$.
Altogether we get
$\uj(l) = \ui(\tau^{-1}(l)) +q(l)$, where $q(l)$ is explicitly known for each $l \in [d]$. In order for $(\uj,\uee_0)$ and therefore, for $(\ui,\uee)$ to be a neutral word, we equate $\uj(l)$ to $\uj(d-l+1)$ for every $l \in [\frac{d}{2}]$ and uniquely determine the values of $\ui(\tau^{-1}(d-l+1))$ for each $l \in [\frac{d}{2}]$. 
 Finally, we have $\tau^{-1}(l) = \tau(\ui,\uee)^{-1}(v_l) = \tau(\ui,\uee)^{-1}(l)$ for all $l \in [d]$, so that $\tau(\ui,\uee) = \tau$ as required.
\end{proof}

We will use Proposition \ref{proposition:spaceddetermineneutral} in the proof of Proposition \ref{proposition:lowerbound}. Let us now look at an example that illustrates the filling in of missing letters.

\begin{example} \label{example:determind}
Suppose $d=8$ and
$\tau = (1,4,7,2) (3,6)(5,8)$. The pair partition in this case is
$\pi = \tau^{-1} (\piR) = \{\{2,5\}, \{4,7\}, \{3, 6\}, \{1,8\}\}$.
Suppose $\uee=(1,1,-1,-1,-1,1,1,-1)$ and $\ui(1) = 0, \ui(2) = 75, \ui(6) = 25$ and $\ui(7) = 50$. We represent the pair partition with the values of $\ui$ that are already known and show that the other values can be filled in uniquely as shown in Proposition \ref{proposition:spaceddetermineneutral}.

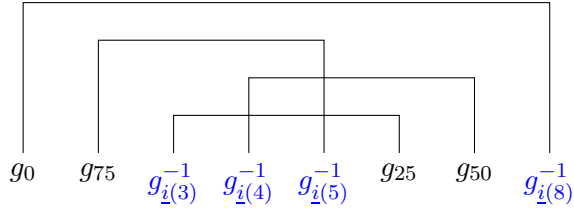
\[ \begin{tikzpicture}[scale=0.5]
    \draw (0,0) node[below] {$g_{0}$} -- (0,4) --  (14,4) -- (14,0)  node[below] {\color{blue} $g_{\ui(8)}^{-1}$};
    \draw (2,0) node[below] {$g_{75}$} -- (2,3) -- (8,3) -- (8,0) node[below] {\color{blue} $g_{\ui(5)}^{-1}$};
    \draw (4,0)  node[below] {\color{blue} $g_{\ui(3)}^{-1}$}-- (4,1) -- (10,1)  -- (10,0) node[below] {$g_{25}$} ;
    \draw  (6,0) node[below] {\color{blue} $g_{\ui(4)}^{-1}$} -- (6,2) -- (12,2) -- (12,0) node[below] {$g_{50}$};
    \end{tikzpicture} \]
Observe that the indices $\ui(1), \ui(2), \ui(6), \ui(7)$ satisfy the inequalities:
\[
\ui(\tau^{-1}(4)) < \ui(\tau^{-1}(3)) - 3d, \ \ui(\tau^{-1}(3)) < \ui(\tau^{-1}(2))-3d, \ \ui(\tau^{-1}(2)) < \ui(\tau^{-1}(1))-3d.
\]
We show explicitly some steps of the reduction $\rightsquigarrow$ with reasoning:
\begin{align*}
(g_0,g_{75},g_{\ui(3)}^{-1},g_{\ui(4)}^{-1},g_{\ui(5)}^{-1},g_{25},g_{50},g_{\ui(8)}^{-1}) & \rightsquigarrow (g_0,g_{75},g_{25},g_{\ui(3)}^{-1},g_{\ui(4)+1}^{-1},g_{\ui(5)+1}^{-1},g_{50},g_{\ui(8)}^{-1}).
\end{align*}
Here we must have $\ui(5) > 25$, $\ui(4) >25$ and $\ui(3) =25$ due to the equations \eqref{equation:spaced} and \eqref{equation:restriction}.
Finally, we get that
the normal form of $(\ui,\uee)$ is
$(\uj,\uee_0) = (g_{74},g_{49},g_{24},g_0,g_{\ui(8)}^{-1}, g_{\ui(3)-1}^{-1},g_{\ui(4)}^{-1},g_{\ui(5)+1}^{-1})$. 
Equating $\uj(l)$ to $\uj(d-l+1)$ for each $l \in [\frac{d}{2}]$ gives
\begin{align*}
    \ui(8) = 0, & \  \ui(3) = 24+1=25 \\
    \ui(4) =49, & \ \ui(5) = 74-1=73.
\end{align*}

Our filled in word is visualized as:

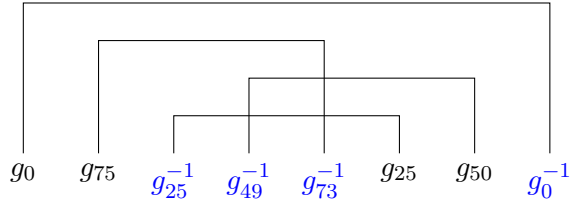
\[ \begin{tikzpicture}[scale=0.5]
    \draw (0,0) node[below] {$g_{0}$} -- (0,4) -- (14,4) --  (14,0) node[below] {\color{blue} $g_{0}^{-1}$}  ;
    \draw (2,0) node[below] {$g_{75}$} -- (2,3) -- (8,3) -- (8,0) node[below] {\color{blue} $g_{73}^{-1}$};
    \draw (4,0) node[below] {\color{blue} $g_{25}^{-1}$} -- (4,1) --  (10,1)  -- (10,0) node[below] {$g_{25}$} ;
    \draw (6,0) node[below] {\color{blue} $g_{49}^{-1}$} -- (6,2) -- (12,2) -- (12,0) node[below] {$g_{50}$};
    \end{tikzpicture} \]

So we have $(\ui,\uee) =  (g_0, g_{75}, g_{25}^{-1}, g_{49}^{-1},g_{73}^{-1}, g_{25}, g_{50}, g_0^{-1})$ and $\tau(\ui,\uee) = \tau$.
\end{example}
\section{Counting the size of a bin} \label{section:count}
In Section \ref{section:algorithm}, we assigned to each word $(\ui,\uee) \in \cW_0(d)$ a unique permutation $\tau(\ui,\uee)$ that records the transformation of $(\ui,\uee)$ into its unique normal form $(\uj,\uee_0)$ under the reduction $\rightarrow$ given by \eqref{equation:neutralnormalform}. Further, the permutation provides us with a ``bin'' -- that is, a pair partition $\pi(\ui,\uee) \in \cP_2(d)$ such that $\pi(\ui,\uee) = \tau(\ui,\uee)^{-1} (\piR)$, where $\piR$ is the rainbow pair partition on $d$ points.
We will now show that for every pair partition $\pi \in \cP_2(d)$, the number of words $(\ui,\uee)$ in $\neutral$ with $\pi(\ui,\uee)=\pi$ can be approximated in the limit as $n \to \infty$. 

We first count the number of permutations associated to each pair partition in this set-up. 
As before, if $\pi=\{\{k_1,l_1\}, \ldots\{k_{\frac{d}{2}}, l_{\frac{d}{2}}\}\}$ and $\tau \in \cS_d$, we will write $\tau(\pi)$ to mean the pair partition given by
\[
\{\{\tau(k_1), \tau(l_1)\}, \ldots\{\tau(k_{\frac{d}{2}}), \tau(l_{\frac{d}{2}})\}\}.
\]

\begin{lemma}\normalfont  \label{lemma:permnumber}
Let $\pi \in \cP_2(d)$ be any pair partition and $\piR \in \cP_2(d)$ be the rainbow pair partition. There exist $d!!$
permutations $\tau \in \cS_d$ such that $\tau(\pi) = \piR$.
\end{lemma}

\begin{proof}
Each pair in the pair partition $\pi$ must be sent to a pair in the rainbow partition by the permutation $\tau$. This gives a total of $d \cdot (d-2) \cdot \cdots 2 =d!!$ permutations $\tau$ such that $\tau(\pi)= \piR$.
\end{proof}

Next, we will show that for each permutation $\tau \in \cS_d$, the number of words $(\ui,\uee) \in \neutral$ such that $\tau(\ui,\uee) = \tau$, (where $\tau(\ui,\uee)$ is the uniquely determined permutation that arises in Lemma \ref{lemma:perm}), is bounded by two polynomials in $n$ of degree $\frac{d}{2}$ and with leading coefficient $(\frac{d}{2})!$. For each $\tau \in \cS_d$, the set whose cardinality we would like to find bounds on is: 
\[\cW_0(d,n,\tau) : = \{(\ui,\uee) \in \neutral \mid \tau(\ui,\uee) = \tau\}.\]
Let
\[N(d,n,\tau) := |\cW_0(d,n, \tau)|.\] 

\begin{remark} \normalfont \label{remark:epsilon}
It is a simple observation from the structure of $\uee_0$ in the normal form $(\uj,\uee_0)$ in \eqref{equation:neutralnormalform} that given a permutation $\tau \in \cS_d$, any word $(\ui,\uee)$ in $\cW_0(d,n,\tau)$ must satisfy the following rule for $\uee$:

\[
l \in \{1,\ldots,\frac{d}{2}\} \implies \uee(\tau^{-1}(l)) =+1; \quad
l \in \{\frac{d}{2}+1,\ldots,d\} \implies \uee(\tau^{-1}(l)) =-1.
\]
\end{remark}

\begin{proposition} \normalfont  \label{proposition:upperbound}
Let $d$ be a fixed positive even integer and $n \in \bN$. Let $\piR \in \cP_2(d)$ denote the rainbow pair-partition and $\pi \in \cP_2(d)$ be any pair-partition. Then for each $\tau \in \cS_d$ with $\tau(\pi) = \piR$, the following inequality holds:

\begin{equation}\label{equation:upperbound}
N(d,n,\tau) \leq 
{\binom{n+\frac{d^2}{2}-2}{\frac{d}{2}}}.
\end{equation}
\end{proposition}

\begin{proof}
Define the set \[\cU(d,n) := \{\uk: [\frac{d}{2}] \to \{0, \ldots, n-1\} \mid \uk(l+1) \leq \uk(l) +d+1, l \in [\frac{d}{2}-1] \}.\]
Note that the set $\cU(d,n)$ is independent of the permutation $\tau$. We will provide an injection from the set $\cW_0(d,n,\tau)$ into $\cU(d,n)$.
Let $(\uj,\uee_0)$ be the unique normal form of a word $(\ui,\uee) \in \cW_0(d,n,\tau)$. Recall that \eqref{equation:ij-inequality} gives the inequality
\[\ui(\tau^{-1}(l))-\frac{d}{2} \leq \uj(l) \leq \ui(\tau^{-1}(l))+\frac{d}{2}, \ l \in [\frac{d}{2}].\]

Furthermore, by \eqref{equation:neutralnormalform}, $\uj$ satisfies
\[\uj(l+1) \leq \uj(l)+1, \ l \in [\frac{d}{2}-1].\]

Hence we get for $(\ui,\uee) \in \cW_0(d,n,\tau)$,
\[
\ui(\tau^{-1}(l+1)) \leq \ui(\tau^{-1}(l)) +d+1
\]

Let
\[
\iota((\ui,\uee)) :=  \uk,
\]
with $\uk(l) = \ui(\tau^{-1}(l))$ for $l\in [\frac{d}{2}]$. Then $\iota$
defines a map from $(\cW_0(d,n,\tau)$ into $U(d,n)$.
Next we show that $\iota$ is injective. Suppose $(\ui,\uee), (\tilde{\ui}, \tilde{\uee}) \in \cW_0(d,n,\tau)$ and $\ui(\tau^{-1}(l)) = \tilde{\ui}(\tau^{-1}(l))$ for each $l \in [\frac{d}{2}]$. By Remark \ref{remark:epsilon} it is clear that $\uee= \tilde{\uee}$. Hence it suffices to show that 
\[
\ui((\tau^{-1}(l))= \tilde{\ui}(\tau^{-1}(l)), \ l \in \{\frac{d}{2}+1, \ldots, d\}.
\]
In other words, it suffices to show that
$\ui(l_-)=\tilde{\ui}(l_-)$ for every pair $\{l_+,l_-\} \in \pi$ given that $\ui(l_+)=\tilde{\ui}(l_+)$. But this is precisely the content of Proposition \ref{proposition:determineneutral}.
Altogether this gives $(\ui,\uee) = (\tilde{\ui},\tilde{\uee})$ as desired.
The inequalities in the definition of $\cU(d,n)$
\[
\uk(l+1) \leq \uk(l) +d+1, \ l \in [\frac{d}{2}-1]
\]
give an injection from $\cU(d,n)$ into the set
$\{\uk:[\frac{d}{2}] \to \{\frac{-d^2}{2}+2,\ldots, n-1\} \mid \uk(l+1) <\uk(l), \ l \in [\frac{d}{2}-1]\}$, whose 
cardinality is given by
\[
\binom{n+\frac{d^2}{2}-2}{\frac{d}{2}}.
\]
Hence $N(d,n,\tau) \leq \binom{n+\frac{d^2}{2}-2}{\frac{d}{2}}$, as required.
\end{proof}

\begin{proposition}\normalfont 
 \label{proposition:lowerbound}
Let $d$ be a fixed positive even integer and $n \in \bN$ be such that $n > \frac{3d}{2}(d-1)$. Let $\piR \in \cP_2(d)$ denote the rainbow pair-partition and $\pi \in \cP_2(d)$ be any pair-partition. Then for each $\tau \in \cS_d$ with $\tau(\pi) = \piR$, the following inequality holds:

\begin{equation}\label{equation:lowerbound}
{\binom{n+2d-\frac{3d^2}{2}}{\frac{d}{2}}}\leq N(d,n,\tau).
\end{equation}
\end{proposition}

\begin{proof}
For $c$ a positive even integer, and $m \in \bN_0$, let
\[\cL(c,r) := \{\uk: [\frac{c}{2}] \to \{0, \ldots, r-1\} \mid \uk(m) < \uk(m-1)-3c, m \in \{2, \ldots, \frac{c}{2}\}\}.\]

Given $\tau \in \cS_d$ and $\uk \in \cL(d, n-d)$, define 
\[
\ui(\tau^{-1}(l)) : =\uk(l), \text{ if } l \in \{1, \ldots, \frac{d}{2}\}.
\]

Define $\uee$ by 
\[
\uee(\tau^{-1}(l)) = \begin{cases}
1 & l \in \{1,\ldots,\frac{d}{2}\} \\
-1 & l \in \{\frac{d}{2}+1,\ldots,d\}.
\end{cases}
\]

As
\[
\ui(\tau^{-1}(l)) < \ui(\tau^{-1}(l-1)) -3d, \ l \in \{2,\ldots, \frac{d}{2}\},
\]
by Proposition \ref{proposition:spaceddetermineneutral}, there exist unique values 
  $\ui(\tau^{-1}(l))$ for $l \in \{\frac{d}{2}+1,\ldots, d\}$ such that $(\ui,\uee) \in \cW_0(d)$ and $\tau(\ui,\uee)= \tau$. Further, by Corollary \ref{corollary:ineq}, as $\ui(\tau^{-1}(l)) = \uk(l)$ takes values in $\{0,\ldots, n-d-1\}$ for $l \in [\frac{d}{2}]$, we must have that $\ui(\tau^{-1}(l))$ takes values in $\{0, \ldots, n-1\}$ for $l \in \{\frac{d}{2}+1,\ldots, d\}$, so that $(\ui,\uee) \in \cW_0(d,n,\tau)$.
This allows us to define $\iota : {\cL(d,n-d)} \to \cW_0(d,n, \tau)$
by $\iota(\uk) = \ui$, with $\ui$ as described above. Indeed, $\iota$ is an injection from $\cL(d,n-d)$ into the set $\cW_0(d,n, \tau)$ because $\iota_1(\uk_1) = \iota(\uk_2)$ implies in particular that  
$\uk_1(l) = \iota(\uk_1)(\tau^{-1}(l)) = \iota(\uk_2)(\tau^{-1}(l)) =  \uk_2(l)$ for all $l \in [\frac{d}{2}]$.

The set $\cL(d,n-d)$ is described by the inequalities
\[
\uk(l+1) < \uk(l) - 3d, \ l\in [\frac{d}{2}-1], 
\]
for tuples $\uk$ taking values in $\{0, \ldots, n-d-1\}$. These inequalities give an injection from the set
$\{\uk: [\frac{d}{2}] \to \{\frac{3d^2}{2}-3d,\ldots, n-d-1\} \mid \uk(l+1) < \uk(l), \ l \in [\frac{d}{2}-1]\}$ into $\cL(d,n-d)$.
Hence $|\cL(d,n-d)|$ is bounded below by
$\binom{n+2d-\frac{3d^2}{2}}{\frac{d}{2}}$.
\end{proof}

\section{A Central Limit Theorem for \texorpdfstring{$\bC(F)$}{C(F)}}\label{section:main}
\subsection{Main Result}
We are now ready to prove our main result, Theorem \ref{theorem:main}, which we restate here for convenience:
\begin{theorem*}[CLT for the sequence $a_n$]
Let  $(a_n)$ be the sequence of self-adjoint random variables in $(\bC(F), \varphi)$ given by
\[
a_n = \frac{g_n+g_n^*}{\sqrt{2}}, \ n \in \bN_0
\]
and
\[
s_n := \frac{1}{\sqrt{n}} (a_0+ \cdots +a_{n-1}), \ n \in \bN.    
\] 

Then we have
\[\lim_{n \to \infty} \varphi (s_n^d) = \begin{cases}
(d-1)!! & \ \text{for } d \text{ even,} \\
0 & d \ \text{for } d \text{ odd.}
\end{cases}\]

That is,
\[
s_n \stackrel{\text{distr}}{\longrightarrow} x,
\]
where $x$ is a normally distributed random variable of variance $1$.
\end{theorem*}

\begin{proof}
Recall the sets $\cW_0(d,n)$ and $\cW_0(d,n,\tau)$ defined as before for $d, n \in \bN:$

\[ \neutral := \Bigl\{  ( \ui , \uee )  \begin{array}{ll}
\vline & \ui : [d] \to \{ 0,1, \ldots , n-1 \}, \ \uee : [d] \to \{ -1,1 \} \\
\vline & \mbox{such that} \ \eval_F(\ui,\uee) = g_{\ui(1)}^{\uee(1)} \cdots g_{\ui(d)}^{\uee(d)} =  e 
\end{array}  \Bigr\}
\]

and for each $\tau \in \cS_d$,

\[
\cW_0(d,n,\tau) = \{(\ui,\uee) \in \neutral \mid \tau(\ui,\uee)=\tau\}.
\]

The moment of order $d$ of $s_n$ is given by
\begin{align*}
 \varphi (s_n^d) &=  \frac{1}{(2n)^{d/2}}
\sum_{ \substack{\ui : [d] \to \{ 0, \ldots , n-1 \}, \\
                  \uee : [d] \to \{ -1,1 \} }  } 
\ \varphi \Bigl( 
g_{\ui (1)}^{\uee (1)} \cdots g_{\ui (d)}^{\uee (d)} \Bigr) \\
 &= \frac{1}{(2n)^{d/2}}
\sum_{ \substack{\ui : [d] \to \{ 0, \ldots , n-1 \}, \\
                  \uee : [d] \to \{ -1,1 \} \\
              \eval_F(\ui,\uee) =e   }  } 
\ \varphi \Bigl( 
g_{\ui (1)}^{\uee (1)} \cdots g_{\ui (d)}^{\uee (d)} \Bigr) \\
&=   \frac{1}{(2n)^{d/2}} |\cW_0(d,n)| \end{align*}
For odd $d \in \bN$, $\cW_0(d,n) = \emptyset$, so $\varphi(s_n^d)=0$ for every $n \in \bN$, and thus $\lim_{n \to \infty} \varphi(s_n^d) =0$.
For even $d \in \bN$, by Propositions \ref{proposition:upperbound} and \ref{proposition:lowerbound}, we have for
each $\tau \in \cS_d$ and for $n > \frac{3d}{2}(d-1)$,
\begin{equation} \label{equation:lowerupper}
{\binom{n+2d-\frac{3d^2}{2}}{\frac{d}{2}}} \leq N(d,n,\tau) \leq 
{\binom{n+\frac{d^2}{2}-2}{\frac{d}{2}}}.\end{equation}
Dividing each term in the inequalities  \eqref{equation:lowerupper} by $n^{\frac{d}{2}}$ and on taking limits as $n\ \to \infty$, we get for every $\tau \in \cS_d$,
\begin{equation} \label{equation:ineq}
\frac{1}{{(\frac{d}{2})!}}
\leq \lim_{n\to \infty} \frac{1}{n^{\frac{d}{2}}} N(d,n,\tau) \leq 
\frac{1}{{(\frac{d}{2})!}}.
\end{equation}

Now
\begin{align*}
\lim_{n\to \infty} \varphi (s_n^d) &= \lim_{n\to \infty}  \frac{1}{(2n)^{d/2}} |\cW_0(d,n)| \\
&= \lim_{n\to \infty}  \frac{1}{(2n)^{d/2}} \sum_{ \pi \in \cP_2(d)} \sum_{\substack{\tau \in \cS_d, \\ \tau(\pi) =\piR}} |\cW_0(d,n,\tau)| \\
&= \lim_{n\to \infty}  \frac{1}{(2n)^{d/2}} \sum_{ \pi \in \cP_2(d)} \sum_{\substack{\tau \in \cS_d,\\ \tau(\pi) =\piR}} N(d,n,\tau) \\
&= \sum_{\pi \in \cP_2(d)} \frac{1}{2^{\frac{d}{2}}}\sum_{\substack{\tau \in \cS_d, \\ \tau(\pi)=\piR} }\lim_{n \to \infty} \frac{1}{n^{\frac{d}{2}}} N(d,n, \tau).
\end{align*}

We get the following inequalities from \eqref{equation:ineq}:
\begin{align*}
\sum_{\pi \in \cP_2(d)} \frac{1}{2^{\frac{d}{2}}} \sum_{\substack{\tau \in \cS_d, \\ \tau(\pi) =\piR}}
\frac{1}{(\frac{d}{2})!}
& \leq \lim_{n \to \infty} \varphi(s_n^d) \\
&=\sum_{\pi \in \cP_2(d)} \frac{1}{2^{\frac{d}{2}}} \sum_{\substack{\tau\in \cS_d, \\ \tau(\pi)=\piR}}\lim_{n \to \infty} \frac{1}{n^{\frac{d}{2}}} N(d,n, \tau) \\
& \leq \sum_{\pi \in \cP_2(d)} \frac{1}{2^{\frac{d}{2}}} \sum_{\substack{\tau \in \cS_d, \\ \tau(\pi)=\piR} }\frac{1}{(\frac{d}{2})!}.
\end{align*}

By Lemma \ref{lemma:permnumber}, given the rainbow pair partition $\piR$ and any pair partition $\pi \in \cP_2(d)$, the number of permutations $\tau \in \cS_d$ with $\tau(\pi) = \piR$ is $d!!=2^{\frac{d}{2}}(\frac{d}{2})!$. Hence
\[
\sum_{\pi \in \cP_2(d)} 1 = \sum_{\pi \in \cP_2(d)} \sum_{\substack{\tau\in \cS_d, \\ \tau(\pi)=\piR}} \frac{1}{2^{\frac{d}{2}}(\frac{d}{2})!} \leq \lim_{n \to \infty} \varphi(s_n^d) \leq \sum_{\pi \in \cP_2(d)} \sum_{\substack{\tau\in \cS_d, \\ \tau(\pi)=\piR}} \frac{1}{2^{\frac{d}{2}}(\frac{d}{2})!}=\sum_{\pi \in \cP_2(d)} 1,
\]
The number of pair partitions $\pi$ in $\cP_2(d)$ is $(d-1)!!$,
so we arrive at
\[\lim_{n \to \infty} \varphi(s_n^d) = |\cP_2(d)| =(d-1)!!.\]
\end{proof}

The following corollary is a consequence of Theorem \ref{theorem:main} and the fact that the normal distribution is determined by its moments. See for instance Example 30.1 and Theorem 30.2 in \cite{Bil95}.

\begin{corollary} \normalfont 
Let $(a_n)$ and $(s_n)$ be the sequences described in Theorem \ref{theorem:main}. For every $n \in \bN$, let $\mu_n$ denote the law of $s_n$. As $n \to \infty$, the probability measures $\mu_n$ have a $w^*$-limit $\mu$, where $\mu= N(0,1)$ is the law of the centered normal distribution of variance $1$. 
\end{corollary}

\subsection{Further questions}
It would be natural to study a multi-dimensional version of the central limit theorem studied here and the combinatorics in that case; however, this question is out of the scope of the current paper. Also of interest is the question of which groups given by infinite presentations lend themselves to a central limit theorem of the type described here. It would finally also be interesting to study convergence rates for the limit theorem, and compare with those for the classical central limit theorem.

\subsection*{Acknowledgements}
The author would like to thank Alexandru Nica for proposing the question of finding a central limit theorem for the Thompson group $F$, and for several fruitful discussions. The author also thanks Claus K\"ostler for the introduction to $F$, and for many helpful discussions. Many thanks to the anonymous referee for their helpful suggestions.


\begin{thebibliography}{99}
\bibitem{AJ21}  Aiello, V.; and Jones, V.~F.~R. (2021).
\newblock On spectral measures for certain unitary representations of R. Thompson's group $F$. 
\newblock \emph{J. Funct. Anal.} 280 (1), 108777, 1--27.
\bibitem{BN99} Baader, F.; and Nipkow, T. (1999).
\newblock Term rewriting and all that. 
\newblock \emph{Cambridge University Press}.

\bibitem{Be04}
Belk J., Thompson's group~{$F$} (2004). \newblock Ph.D.~Thesis, Cornell University.
\newblock \url{http://dml.mathdoc.fr/item/0708.3609}.

\bibitem{Bi95} Biane, P. (1995)
\newblock Permutation model for semi-circular systems and quantum random walks.
\newblock \emph{Pacific Journal of Mathematics} 171 (2), 373 -- 387.

\bibitem{Bil95} Billingsley, P. (1995)
\newblock Probability and Measure. 
\newblock 3rd Edition, Wiley and Sons.

\bibitem{Br20}
Brothier, A.  (2020) 
\newblock On Jones’ connections between subfactors, conformal field theory, Thompson’s groups and knots. \newblock \emph{Celebratio Mathematica, Vaughan FR Jones Volume}.

\bibitem{BJ19a}
Brothier A.; and Jones V.F.R. (2019) \newblock Pythagorean representations of {T}hompson's groups. 
\newblock \emph{Journal of Functional Analysis} 277, 2442--2469,


\bibitem{BS94} Bo\.{z}ejko, M.; and Speicher R. (1994) \newblock Interpolations between bosonic and fermionic relations given by generalized brownian motions.
\newblock \emph{Mathematische Zeitschrift} 222, 135-160.


\bibitem{CKN22} Campbell, J.; K\"ostler C., and Nica A. (2022) 
\newblock A central limit theorem for star-generators of $\cS_{\infty}$, which relates to traceless CCR-GUE matrices.
\newblock \emph{International Journal of Mathematics} 33 (09), 2250065.

\bibitem{CF11}
Cannon J.W., Floyd W.J., What is {$\dots$} {T}hompson's group?, \textit{Notices
 Amer. Math. Soc.} \textbf{58} (2011), 1112--1113.

\bibitem{CFP96}
Cannon J.W., Floyd W.J., Parry W.R., Introductory notes on {R}ichard
 {T}hompson's groups, \textit{Enseign. Math.} \textbf{42} (1996), 215--256.


\bibitem{DT19} Dehornoy, P.; and Tesson. E. (2019).
\newblock Garside combinatorics for Thompson's monoid $F^+$ and a hybrid with the braid monoid $B_{\infty}^{+}$.
\newblock \emph{Algebraic Combinatorics} 2 (4), 683--709.

\bibitem{KK22}
K\"ostler, C.; and Krishnan, A. (2022).
\newblock Markovianity and the Thompson group $F$. 
\newblock \emph{SIGMA} 18 (83), 27 pages. 

\bibitem{KN21} K\"ostler, C.; and Nica, A. (2021)
\newblock A central limit theorem for star-generators of $\cS_{\infty}$, which relates to the law of a GUE matrix.
\newblock \emph{Journal of Theoretical Probability} 34 (3), 1248 -- 1278.

\bibitem{Ne42} Newman, M. H. A. (1942).
\newblock On theories with a combinatorial definition of  ``equivalence''. 
\newblock \emph{Annals of Mathematics} 223--243.

\bibitem{NS06} 
Nica, A.; and Speicher, R. (2006)
\newblock Lectures on the combinatorics of free probability.
\newblock \emph{Cambridge University Press}.

\bibitem{Sk22}
Skeide, M. (2022).
\newblock Algebraic central limit theorems: A personal view on one of Wilhelm’s legacies.
\newblock \emph{Infinite Dimensional Analysis, Quantum Probability and Related Topics} 25 (4), 224003. 

\bibitem{Sp90} Speicher, R. (199)
\newblock A new example of ``independence'' and ``white noise''.
\newblock \emph{Probability theory and related fields} 84 (2), 141--159.

\bibitem{SW94} Speicher, R.; and von Waldenfels, W. (1994)
\newblock A general central limit theorem and invariance principle.
\newblock \emph{Quantum Probability And Related Topics: QP-PQ} (Volume IX) 371 -- 387.

\end{thebibliography}

\end{document}